\documentclass[reqno]{amsart}
\usepackage[a4paper, margin=1in]{geometry}

\usepackage{hyperref}
\usepackage{amsmath,amsthm,amssymb,mathtools}
\usepackage{bm}
\usepackage{bbm}
\usepackage{booktabs}
\usepackage{graphicx}
\usepackage[nameinlink,capitalize]{cleveref}
\hypersetup{colorlinks=true,linkcolor=blue,citecolor=blue,urlcolor=blue}
\usepackage{microtype}
\usepackage{enumitem}
\usepackage{mathdots}

\usepackage{nicematrix}

\usepackage{matlab-prettifier}
\usepackage{algorithm}
\usepackage[noend]{algpseudocode}

\usepackage{todonotes}

\usepackage[numbers]{natbib}
\usepackage{listings}

\newtheorem{theorem}{Theorem}[section]
\newtheorem{lemma}[theorem]{Lemma}
\newtheorem{proposition}[theorem]{Proposition}
\newtheorem{corollary}[theorem]{Corollary}

\theoremstyle{definition}
\newtheorem{definition}[theorem]{Definition}

\newtheorem{example}[theorem]{Example}

\newtheorem{problem}[theorem]{Problem}
\newtheorem{question}[theorem]{Question}
\newtheorem{remark}[theorem]{Remark}

\newcommand{\C}{\mathbb{C}}
\newcommand{\R}{\mathbb{R}}
\newcommand{\Z}{\mathbb{Z}}

\newcommand{\vect}{\mathrm{vec}}
\newcommand{\rank}{\mathrm{rank}}

\numberwithin{equation}{section}

\begin{document}

\title{Structured Matrix Factorization Length}

\author{Jeong-Hoon Ju}
\address{
\parbox{\linewidth}{(Jeong-Hoon Ju) Department of Mathematical Sciences, University of Copenhagen, 5 Universitetsparken, 2100 Copenhagen Ø, Denmark}}
\email{jeonghoon.ju@math.ku.dk}

\author{Taehyeong Kim*}
\address{
\parbox{\linewidth}{(Taehyeong Kim) Nonlinear Dynamics and Mathematical Application Center, Kyungpook National University, Daegu 41566, Republic of Korea}}
\email{thkim0519@knu.ac.kr}

\thanks{*Corresponding author}

\date{\today}

\begin{abstract}
Every (resp. a generic) complex $n \times n$ matrix can be expressed as a product of $2n+5$ (resp. $\lfloor n/2 \rfloor +1$) Toeplitz matrices. Motivated by this result, it is natural to ask the following question: what is the minimum number of Toeplitz matrices required to factor a given matrix? We generalize this question from Toeplitz structure to more general structures. In this paper, we introduce the notion of structured matrix factorization length when the set of matrices with a given structure is an affine variety $X \subseteq \mathbb{C}^{n \times n}$. Then we introduce the $r$-th $X$-factorization variety, defined as the Zariski closure of the set of products of $r$ matrices in $X$, and use it to define the border structured matrix factorization length. In particular, we study the cases in which $X$ is the affine variety of Toeplitz, Hankel, bidiagonal, tridiagonal, skew-symmetric or companion matrices. We calculate the dimension of the $X$-factorization varieties for all these cases, and discuss how numerical algebraic geometry can be used to obtain computational evidence for the degrees of $X$-factorization varieties with an example. In addition, we propose methods for deriving lower and upper bounds for (border) structured matrix factorization length. For lower bounds, we develop a method based on displacement rank, which can also be used to obtain some defining equations of the $r$-th $X$-factorization variety; for upper bounds, we suggest an approach using alternating minimization. 
\end{abstract}

\keywords{Structured matrix factorization, factorization length, factorization varieties, displacement rank, alternating minimization}

\subjclass[2020] {14A10, 15A23, 15B05}

\maketitle

\tableofcontents

\section{Introduction}

A Toeplitz matrix $T$ is a structured matrix defined by a square matrix whose entries are constant along each diagonal, that is, an $n \times n$ Toeplitz matrix is of the form
\begin{equation}\label{eq:Toeplitz}
    T = 
    \begin{bmatrix}
        a_{0} & a_{1} & \cdots & a_{n-1}\\
        a_{-1} & a_{0} & \ddots & \vdots \\
        \vdots & \ddots & \ddots & a_{1}\\
        a_{-(n-1)} & \cdots & a_{-1} & a_{0}\\
    \end{bmatrix},
\end{equation}
for some constants $a_{-(n-1)},...,a_{n-1}$. Because of this structure, a Toeplitz matrix has computational advantages in tasks such as matrix-vector multiplication and determinant computation. If we embed a given $n \times n$ Toeplitz matrix into a $2n \times 2n$ circulant matrix and use fast Fourier transform (FFT), then the matrix-vector multiplication for a Toeplitz matrix can be calculated in $O(n\log n)$ time rather than $O(n^2)$ (see \cite[P4.8.6]{golub2013matrix}). This implies that $n \times n$ Toeplitz-arbitrary matrix multiplication can be calculated in $O(n^2\log n)$ rather than $O(n^3)$. In addition, the determinant of an $n \times n$ Toeplitz matrix can be computed in $O(n^2)$ time stably rather than $O(n^3)$ (see \cite{bareiss1969numerical, bojanczyk1995stability}). 

Motivated by these results, it is proved that every (resp. generic) complex $n \times n$ matrix can be expressed as a product of $2n+5$ (resp. $\lfloor n/2 \rfloor+1$) Toeplitz matrices in \cite{ye2016every}. Considering the computational advantages of the Toeplitz matrix, it is natural to have the following question: for a given square matrix $A \in \mathbb{C}^{n \times n}$, what is the smallest number $r$ such that $A=T_1\cdots T_r$ where $T_1,...,T_r$ are Toeplitz matrices? On the other hand, there are similar results with respect to other structured matrices, for example, symmetric matrix \cite{bosch1986factorization}, Hankel matrix \cite{ye2016every}, bidiagonal, tridiagonal, skew-symmetric, and companion matrices \cite{ye2017new}. Hence, it is natural to ask the question above for other structures:

\begin{question}
    For a given structure $X$ for square matrices and a given square matrix $A \in \mathbb{C}^{n \times n}$, what is the smallest number $r$ such that $A=M_1\cdots M_r$ where $M_1,...,M_r$ are $X$-matrices?
\end{question}

In the case of Hankel matrix structure, as in the Toeplitz case, this question is also related to the computational complexity of matrix-vector and determinant computation, because a Hankel matrix
\begin{equation*}
    H = 
    \begin{bmatrix}
        a_{-(n-1)} & \cdots & a_{-1} & a_{0}\\
        \vdots & \iddots & a_{0} & a_{1}\\
        a_{-1} & \iddots & \iddots & \vdots \\
        a_{0} & a_{1} & \cdots & a_{n-1}\\
    \end{bmatrix}
\end{equation*}
whose entries are constant along each anti-diagonal, can be obtained by multiplying the reversal permutation matrix
\begin{equation}\label{eq:reversal permutation matrix}
    J:=\begin{bmatrix}
        0 & 0 & \cdots & 1\\
        \vdots & \vdots & \iddots & \vdots\\
        0 & 1 & \cdots & 0\\
        1 & 0 & \cdots & 0
    \end{bmatrix}
\end{equation} 
to the Toeplitz matrix $T$ given in (\ref{eq:Toeplitz}).

In the case of companion matrix structure, a companion matrix is of the form
\begin{equation*}
    C = 
    \begin{bmatrix}
        0 & 0 & \cdots & 0 & c_1\\
        1 & 0 & \cdots & 0 & c_2\\
        \vdots & \vdots & \ddots & \vdots & \vdots \\
        0 & 0 & \cdots & 1 & c_n
    \end{bmatrix}.
\end{equation*}
Hence, multiplication by a companion matrix $C$ shifts the coordinates and adds a rank one feedback term determined by the last coordinate. From the perspective of linear dynamical systems, factoring a transformation $A$ into a product of companion matrices $A = C_1 \cdots C_r$ conceptually decomposes a global transformation $y = Ax_0$ into a sequence of intermediate states $x_k = C_k x_{k-1}$, where each step is an elementary shift-feedback update. Thus, the minimum number of such $r$ measures the minimum number of such elementary updates needed to realize a given linear transformation.

A tridiagonal matrix
\begin{equation*}
    D = 
    \begin{bmatrix}
        a_1 & b_1 & 0 & \cdots & 0\\
        c_1 & a_2 & b_2 & \ddots & \vdots\\
        0 & c_2 & \ddots & \ddots & 0\\
        \vdots & \ddots & \ddots & a_{n-1} & b_{n-1}\\
        0 & \cdots & 0 & c_{n-1} & a_n
    \end{bmatrix}
\end{equation*}
only couples each coordinate with its nearest neighbors. Hence, it represents a linear operator with a local interaction pattern.
Hence, the minimum $r$ such that a given matrix $A$ is factored into a product of tridiagonal matrices $A=M_1\cdots M_r$ measures how many local factors are needed to represent $A$. 

The main contribution of this paper is as follows. To the best of our knowledge, this paper is the first to introduce the notion of structured matrix factorization length and formulate it as an algebro-geometric object. More precisely, for an affine variety $X$ in $\mathbb{C}^{n \times n}$ modeling a matrix structure, we define the \emph{$X$-factorization length} of a matrix $A$ as the smallest number of factors from $X$ needed to express $A$ (see Definition \ref{def:structured factorization length}). We also introduce the \emph{$r$-th $X$-factorization variety} as the Zariski closure of the set of products of $r$ matrices in $X$, and define the \emph{border $X$-factorization length} of a given matrix $A$ as the minimum $r$ such that $A$ is in the $r$-th $X$-factorization variety (see Definition \ref{def:border structured factorization length}). This viewpoint allows us to study structured matrix factorization using tools from algebraic geometry. We develop this framework for several classical matrix structures, including Toeplitz, Hankel, bidiagonal, tridiagonal, skew-symmetric, and companion matrices. For these examples, we compute the dimensions of the corresponding factorization varieties, and we determine the varieties for some cases. We also illustrate how numerical algebraic geometry can be used to obtain computational evidence for their degrees using an example of Toeplitz matrix structure. Furthermore, we introduce a method for lower bounds of border $X$-factorization length, which is based on \emph{displacement rank}. This method for lower bound can also be used to obtain some defining equations of the $X$-factorization varieties. In addition, we study the case of traceless symmetric matrix structure, because it is an orthogonally invariant matrix structure which has not been studied in this field. Finally, we suggest an approach for upper bounds of $X$-factorization length, which uses \emph{alternating minimization}.

This framework is parallel to the well-established theory of tensor rank and secant varieties. That is, we will investigate the analogues of tensor rank, border rank, maximal rank, generic rank, invariance of tensor rank under group actions, flattening methods for lower bounds on border rank, and optimization methods for the upper bounds on tensor rank. We refer to \cite{kolda2009tensor, landsberg2011tensors} for these contents.

The structure of this paper is as follows: In Section \ref{section:Preliminaries}, we introduce some basic notions in algebraic geometry needed in this paper, and review previous relevant results on structured matrix factorization. In Section \ref{section:Structured Matrix Factorization Length}, we define the structured matrix factorization length and border structured matrix factorization length, and we provide some examples of them. In Section \ref{section:Geometry of Structured Matrix Factorizations}, we study the geometry of $r$-th $X$-factorization varieties focusing on their dimensions, degrees, and defining equations. In Section \ref{section:Orthogonally Invariant Matrix Structures}, we investigate the (border) $X$-factorization length when $X$ is orthogonally invariant, in particular, when $X$ is the affine variety of traceless symmetric matrices. Finally, in Section \ref{section:Alternating Minimization for Structured Matrix Factorization}, we discuss an alternating minimization approach for constructing structured matrix factorizations.

\section*{Acknowledgements}

J.-H. J. was supported by the Basic Science Program of the NRF of Korea (NRF-2022R1C1C1010052), the Basic Research Laboratory (grant MSIT no. RS-202400414849) and the Novo Nordisk Foundation (grant NNF20OC0059939). T.~Kim was supported by the National Research Foundation of Korea (NRF) grant funded by the Korea government (MSIT) (No. 2022R1A5A1033624 \&\,RS-2024-00342939 \&\,RS-2025-25436769). The authors would like to thank Alexander Blomenhofer, Matthias Christandl, Thomas Fraser, Yifan Jia, Yeongrak Kim, Itai Leigh, and Maksym Zubkov for helpful discussions and suggestions. 

\section{Preliminaries}\label{section:Preliminaries}

In this section, we briefly introduce some basic notions in algebraic geometry that will be used in this paper. Then we introduce some previous relevant results motivating the notion of structured matrix factorization length. Because we will develop the theory from an algebro-geometric perspective, we let the base field be the complex number field $\mathbb{C}$, unless stated otherwise.

\subsection{Basic algebraic geometry}

Algebraic geometry studies zero loci of polynomial equations. 
When polynomials $f_1,...,f_s$ are given in a polynomial ring $\mathbb{C}[x_1,...,x_n]$, the zero locus of these polynomials
\begin{equation*}
    \mathbf{V}(f_1,...,f_s):=\{p \in \mathbb{C}^n~|~f_1(p)=\cdots=f_s(p)=0\}
\end{equation*}
in $\mathbb{C}^n$ is called the \emph{affine variety defined by $f_1,...,f_s$}. If $I \subseteq \mathbb{C}[x_1,...,x_n]$ is an ideal, then the zero locus
\begin{equation*}
    \mathbf{V}(I):=\{p \in \mathbb{C}^n~|~f(p)=0~\text{ for all } f\in I\}
\end{equation*}
in $\mathbb{C}^n$ is called the \emph{affine variety defined by $I$}. By Hilbert Basis Theorem, an ideal $I$ in $\mathbb{C}[x_1,...,x_n]$ is always finitely generated. Hence, we can say that an affine variety is always a zero locus of finitely many polynomials.

The collection of all affine varieties in $\mathbb{C}^n$ is closed under arbitrary intersection, and hence form the closed sets of a topology on $\mathbb{C}^n$. We call it the \emph{Zariski topology}. When a subset $S$ in $\mathbb{C}^n$ is given, we can consider its closures with respect to the Zariski topology and the Euclidean topology, which are called the \emph{Zariski closure} and \emph{Euclidean closure}, respectively. We will denote them $\overline{S}^{\text{Zar}}$ and $\overline{S}^{\text{Euc}}$, respectively. In general, the Euclidean closure is contained in the Zariski closure, and they do not need to be the same. For example, the Zariski closure of the subset $\mathbb{Z}$ in $\mathbb{C}$ is the whole space $\mathbb{C}$, but the Euclidean closure of $\mathbb{Z}$ is $\mathbb{Z}$ itself.

If $X$ is an affine variety in $\mathbb{C}^n$, then the set
\begin{equation*}
    \mathbf{I}(X):=\{f \in \mathbb{C}[x_1,...,x_n]~|~f(x)=0~\text{ for all }~x \in X\}
\end{equation*}
is an ideal in $\mathbb{C}[x_1,...,x_n]$, which is called the \emph{defining ideal of $X$}. An element in this ideal is called a defining equation of $X$. The quotient ring
\begin{equation*}
    \mathbb{C}[X]:=\mathbb{C}[x_1,...,x_n]/\mathbf{I}(X)
\end{equation*}
is the ring of polynomial functions on $X$, and is called the \emph{coordinate ring} of $X$.

Let two affine varieties $X \subseteq \mathbb{C}^n$ and $Y \subseteq \mathbb{C}^m$ be given. A map $\varphi:X \rightarrow Y$ is called a \emph{morphism} if $\varphi=(\varphi_1,...,\varphi_m)$ where $\varphi_i \in \mathbb{C}[X]$ for all $i=1,...,m$. For example, for an affine variety $X$ in $\mathbb{C}^{n \times n}$, its $r$-product $X^{\times r}$ is also an affine variety in $(\mathbb{C}^{n \times n})^{\times r}$, and the matrix multiplication map 
\begin{equation}\label{eq:matrix multiplication map}
    m_r:X^{\times r} \rightarrow \mathbb{C}^{n \times n}, (M_1,...,M_r) \mapsto M_1\cdots M_r
\end{equation}
is a morphism. This morphism is a central object in this paper. A morphism between affine varieties is called an \emph{isomorphism} if it is bijective and its inverse is also a morphism. If two affine varieties $X$ and $Y$ are isomorphic, we will write $X \cong Y$. If the image of a morphism $\varphi:X \rightarrow Y$ between affine varieties is dense in $Y$ with respect to the Zariski topology, that is, $\overline{\varphi(X)}^{\text{Zar}}=Y$, then $\varphi$ is called a \emph{dominant morphism}.

The image of a morphism between affine varieties need not be Zariski closed, and this leads to the notion of a \emph{constructible set}. Let $X$ be an affine variety. A subset of $X$ is called a \emph{locally closed subset} of $X$ if it is the intersection of an open subset and a closed subset of $X$. A subset of $X$ is called a \emph{constructible subset} of $X$ if it can be written as a finite union of locally closed subsets of $X$. It is well-known that the image of a morphism between affine varieties is a constructible set. In our setting, the image of a matrix multiplication map is a constructible set and need not be Zariski closed, and hence one is naturally led to study its Zariski closure. Moreover, a product of two irreducible affine varieties is also irreducible.

A nonempty subset $S$ of an affine variety $X$ is said to be \emph{irreducible} if it cannot be expressed as the union $S=S_1 \cup S_2$ of two proper subsets which are closed in $S$. It is well-known that, if $\varphi:X \rightarrow Y$ is a morphism between affine varieties and $X$ is irreducible, then the image of $\varphi$ is also irreducible (as a constructible subset) in $Y$. If the defining ideal $\mathbf{I}(X)$ of an affine variety $X$ is a prime ideal, then $X$ is irreducible. Hence, every affine linear subspace, which is defined by affine linear polynomials, in $\mathbb{C}^n$ is an irreducible affine variety.

We will frequently use the notion of a \emph{generic property} in this paper. We say that a property $\mathcal{P}$ holds \emph{generically} on an affine variety $X$ if there exists a proper Zariski closed subset $Y$ of $X$ such that $\mathcal{P}$ holds for every point in $X \setminus Y$. In this situation, we also often say that a \emph{generic element} of $X$ satisfies the property $\mathcal{P}$. For example, a generic $n \times n$ matrix is of full rank, because the set of $n \times n$ matrices which are not of full rank is the affine variety $Y$ defined by the $n \times n$ determinant polynomial, that is, the full-rankness holds for every point in $\mathbb{C}^{n \times n} \setminus Y$.

Two numerical invariants (under isomorphism) play a particularly important role in algebraic geometry, and they are \emph{dimension} and \emph{degree}. Let $X$ be an affine variety in $\mathbb{C}^n$. The \emph{dimension} of $X$, denoted by $\dim X$, is the maximum of $d$ such that there exists a chain $Z_0 \subsetneq Z_1 \subsetneq \cdots \subsetneq Z_d$ of irreducible closed subsets of $X$. Note that the empty set is not irreducible. For example, $\dim \mathbb{C}^n=n$, and the dimension of a single-point set is zero. The codimension of $X$ is $n-\dim X$. Assume that $\mathbf{I}(X)$ is generated by $f_1,...,f_s \in \mathbb{C}[x_1,...,x_n]$. It is well-known that, for the \emph{Jacobian matrix at $p$}
\begin{equation*}
    \mathbf{J}_p=\left[\frac{\partial f_i}{\partial x_j}(p)\right]_{1 \leq i \leq s,~ 1 \leq j \leq n},
\end{equation*}
the dimension of $X$ is the same as
\begin{equation*}
    \min\{\dim(\ker(\mathbf{J}_p))~|~p \in X\}
\end{equation*}
if $X$ is irreducible. Let $p \in X$. The set
\begin{equation*}
    T_pX:=\ker(\mathbf{J}_p)
\end{equation*}
is called the \emph{tangent space of $X$ at $p$}. If $\dim (\ker(\mathbf{J}_p))=\dim X$, then $p$ is called a \emph{smooth point of $X$}. If all points in $X$ are smooth, then $X$ is called a \emph{smooth affine variety}. In this case, $X$ can be regarded as a complex manifold, and $\dim X$ is the same as the dimension of $X$ as a complex manifold. When $X$ is defined as the Zariski closure of the image of a morphism between affine varieties, we will calculate $\dim X$ using the \emph{differential} of the morphism (see Section \ref{subsection:Dimension}). When the morphism is $\varphi=(\varphi_1,...,\varphi_m):X (\subseteq \mathbb{C}^n) \rightarrow Y (\subseteq \mathbb{C}^m)$, the \emph{differential of $\varphi$ at $x \in X$}, denoted by $d\varphi_x$, is defined by the map 
\begin{equation*}
    d\varphi_x:T_xX \rightarrow T_{\varphi(x)}Y, ~(v_1,...,v_n) \mapsto \left(\sum_{i=1}^n \frac{\partial \varphi_1}{\partial x_i}(x)v_i,\cdots , \sum_{i=1}^n\frac{\partial \varphi_m}{\partial x_i}(x)v_i\right).
\end{equation*}
For example, the matrix multiplication map $m_r$ defined at (\ref{eq:matrix multiplication map}) has the differential $d(m_r)_{(M_1,...,M_r)}:T_{M_1}X \times \cdots \times T_{M_r}X \rightarrow T_{M_1\cdots M_r}\mathbb{C}^{n \times n}$ at $(M_1,...,M_r) \in X^{\times r}$ given by
\begin{equation*}
    d(m_r)_{(M_1,...,M_r)}(Z_1,...,Z_r)=\sum_{i=1}^rM_1\cdots M_{i-1}Z_iM_{i+1}\cdots M_r.
\end{equation*}

Let $X \subseteq \mathbb{C}^n$ be an affine variety of dimension $d$. Then the \emph{degree} of $X$, denoted by $\deg X$, is the number of intersection points of $X$ with a generic affine linear subspace of codimension $d$, counted with multiplicity, and it is well-defined. Intuitively, degree measures quantitatively how nonlinear the variety is. For example, an affine linear subspace in $\mathbb{C}^n$ has the degree $1$, and $\mathbf{V}(y-x^2) \subseteq \mathbb{C}^2$ has the degree $2$. Although we do not discuss projective varieties or their degrees in detail, the degree of $X$ agrees with the degree of its projective closure, that is, the smallest projective variety containing $X$. Therefore, in principle, the degree of $X$ can be computed from the Hilbert polynomial of the projective closure of $X$. However, in this paper, we will remain entirely within the framework of affine varieties.

We refer to \cite{MR4952933, hartshorne2013algebraic, shafarevich1994basic} for the details of the contents above.

\subsection{Previous relevant results}

We review the previous results on the existence of structured matrix factorizations for symmetric (see \cite{bosch1986factorization}), Toeplitz, Hankel (see \cite{ye2016every}), bidiagonal, tridiagonal, skew-symmetric, and companion structures (see \cite{ye2017new}) in chronological order.

\begin{theorem}\label{thm:previous results}
    Let $n$ be an arbitrary integer at least $2$. The following structured matrix factorizations exist:
    \begin{itemize}
        \item [(\romannumeral 1)] Every square matrix is a product of two symmetric matrices.
        \item [(\romannumeral 2)] Every (resp. a generic) $n \times n$ matrix is a product of $2n+5$ (resp. $\lfloor n/2 \rfloor+1$) Toeplitz matrices.
        \item [(\romannumeral 3)] Every (resp. a generic) $n \times n$ matrix is a product of $2n+5$ (resp. $\lfloor n/2 \rfloor+1$) Hankel matrices.
        \item [(\romannumeral 4)] A generic $n \times n$ upper (resp. lower) triangular matrix is a product of $n-1$ upper (resp. lower)  bidiagonal matrices.
        \item [(\romannumeral 5)] Every (resp. a generic) $n \times n$ matrix is a product of $8(n-1)$ (resp. $2(n-1)$) tridiagonal matrices.
        \item [(\romannumeral 6)] Every (resp. a generic) $2n \times 2n$ matrix is a product of $r$ skew-symmetric matrices, where (a) $r=21$ (resp. $r= 5$) if $n=2$; (b) $r=17$ (resp. $r=4$) if $n=3$; and (c) $r=13$ (resp. $r=3$) if $n \geq 4$.  On the other hand, a generic  $(2n+1) \times (2n+1)$ singular matrix is a product of $r$ skew-symmetric matrices, where (a) $r=4$ if $n=1$; and (b) $r=3$ if $n \geq 2$.
        \item [(\romannumeral 7)] A generic $n \times n$ matrix is a product of $n$ companion matrices. Every $n \times n$ matrix is a product of $4n$ companion matrices and a diagonal matrix.
    \end{itemize}


\end{theorem}

The proof idea for this theorem is as follows. Let $X$ be the affine variety of matrices with a given structure. Consider a well-defined matrix multiplication map $m_r:X^{\times r} \rightarrow Y$, where $Y$ is an affine variety in $\mathbb{C}^{n \times n}$. Note that $Y$ would be the variety of $n \times n$ upper (or lower) triangular matrices and the variety of $(2n+1) \times (2n+1)$ singular matrices for the statements (\romannumeral4) and (\romannumeral6) in Theorem \ref{thm:previous results}, respectively. In addition, $Y$ would be $\mathbb{C}^{n \times n}$ for the other statements in Theorem \ref{thm:previous results}. We have the following proposition:

\begin{proposition}[see {\cite[Proposition 3.7]{ye2017new}}]\label{prop:the rank of differential and dominance}
    Assume that $\dim (X^{\times r}) (=r \cdot \dim X) \geq \dim Y$. If there exists a point $(M_1,...,M_r) \in X^{\times r}$ such that the differential $(dm_r)_{(M_1,...,M_r)}:T_{M_1}X \times \cdots \times T_{M_r}X \rightarrow T_{M_1\cdots M_r}Y$ has full rank $\dim Y$, then the map $m_r$ is dominant.
\end{proposition}

Hence, in order to show that a generic matrix in $Y$ is a product of $r$ matrices in $X$, that is, the matrix multiplication map $m_r: X^{\times r} \rightarrow Y$ is dominant, it would suffice to find a point $(M_1,...,M_r) \in X^{\times r}$ such that $(dm_r)_{(M_1,...,M_r)}$ has full rank $\dim Y$. If $Y$ is a topological group such as the general linear group $\operatorname{GL}_n$ or its Borel subgroup $\mathbb{B}_n$ (i.e., the set of invertible upper triangular matrices), then we can use the following lemma to extend such a generic property to all elements in $Y$:

\begin{lemma}[see {\cite[Proposition 1.1.3.(a)]{borel2012linear} or \cite[Proposition 3.8]{ye2017new}}]\label{lemma:topological group}
    A topological group $G$ can be expressed as $G=U \cdot U:=\{hh'~|~h,h' \in U\}$ for an open dense subset of $G$ containing the identity element of $G$.
\end{lemma}

For example, when it is known that a generic $n \times n$ matrix is a product of $\lfloor n/2 \rfloor +1$ Toeplitz matrices, Lemma \ref{lemma:topological group} implies that every invertible $n \times n$ matrix is a product of $2(\lfloor n/2 \rfloor +1)$ Toeplitz matrices (see \cite[Corollary 2]{ye2016every}). In addition, when it is known that a generic $n \times n$ upper triangular matrix is a product of $n-1$ upper bidiagonal matrices, this lemma implies that every invertible $n \times n$ upper triangular matrix is a product of $2(n-1)$ upper bidiagonal matrices (see \cite[Corollary 5.2]{ye2017new}). In addition, if $Y=\mathbb{C}^{n \times n}$, we can use the following proposition in order to extend a generic property to all matrices in $\mathbb{C}^{n \times n}$:

\begin{proposition}[see {\cite[Proposition 3.10]{ye2017new}}]\label{prop:the number of structured matrices needed to express every matrix}
    Assume that $X$ is an irreducible affine variety in $\mathbb{C}^{n \times n}$ and the map $m_r:X^{\times r} \rightarrow \mathbb{C}^{n \times n}$ is dominant. Then every $n \times n$ matrix is a product of $4r$ matrices in $X$ and a diagonal matrix. In particular, if $X$ contains all diagonal matrices, every $n \times n$ matrix is a product of $4r+1$ matrices in $X$.
\end{proposition}

Note that the original statement in \cite[Proposition 3.10]{ye2017new} assumes that $X$ is a linear subspace of $\mathbb{C}^{n \times n}$, but the author only used the condition that $X$ is irreducible (more precisely, the condition that $X^{\times r}$ is irreducible). This proposition was proved by using Lemma \ref{lemma:topological group}. Since the intersection of the dense subset $m_r(X^{\times r})\subseteq \mathbb{C}^{n \times n}$ and the general linear group $\operatorname{GL}_n$ contains an open dense subset of $\operatorname{GL}_n$ (we need the irreducibility of $X$ here), then we obtain that every $n \times n$ invertible matrix is a product of $2r$ matrices in $X$. Furthermore, since every matrix $A \in \mathbb{C}^{n \times n}$ can be expressed as $A=PDQ$ where $P,Q \in \operatorname{GL}_n$ and $D$ is a diagonal matrix, then we obtain the proposition. All extensions of a generic property described in Theorem \ref{thm:previous results} are proved by this idea. Note that every $n \times n$ matrix is a product of $8(n-1)$ (rather than $8(n-1)+1$) tridiagonal matrices, because a product of a diagonal matrix and a tridiagonal matrix is tridiagonal.

A natural question is how to predict a number $r$ such that $m_r:X^{\times r} \rightarrow Y$ is dominant. The following proposition gives lower bounds for the number $r$:

\begin{proposition}[see {\cite[Corollary 3.4 and Corollary 3.6]{ye2017new}}]\label{prop:lower bound of generic length}
    Assume that $X$ and $Y$ are irreducible affine varieties, and $m_r:X^{\times r} \rightarrow Y$ is dominant. Let $\dim X=d$. Then
    \begin{equation*}
        r \geq \left\lceil \frac{\dim Y}{d} \right\rceil.
    \end{equation*}
    If $X$ is a cone (i.e., $x \in X$ implies that $\lambda x \in X$ for every $\lambda \in \mathbb{C}$), then
    \begin{equation*}
        r \geq \left\lceil \frac{\dim Y-1}{d-1} \right\rceil.
    \end{equation*}
\end{proposition}

This proposition is a corollary of Lemma \ref{lemma:fiber dimension} below, which we will use when investigating the dimension of our target varieties.

In summary, the strategy used to prove Theorem \ref{thm:previous results} is as follows:
\begin{itemize}
    \item [Step 1)] Calculate a lower bound $l$ of the smallest number $r'$ such that $m_{r'}:X^{\times r'} \rightarrow Y$ is dominant using Proposition \ref{prop:lower bound of generic length}, if both $X$ and $Y$ are irreducible.
    \item [Step 2)] Consider a number $r \geq l$, and calculate the rank of the differential $d(m_{r})_p$ at a point $p \in X^{\times r}$. If the rank is equal to $\dim Y$, then $m_{r}$ is dominant by Proposition \ref{prop:the rank of differential and dominance}, and hence a generic matrix in $Y$ is a product of $r$ matrices in $X$.
    \item [Step 3)] If $Y=\mathbb{C}^{n \times n}$ and $X$ is irreducible, every $n \times n$ matrix is a product of $4r+1$ matrices in $X$ by Proposition \ref{prop:the number of structured matrices needed to express every matrix}.
\end{itemize}
We will also use this strategy to show such statements in this paper (see Theorem \ref{thm:dimension for tridiagonal structure} and Theorem \ref{thm:the generic length of traceless symmetric matrix structure}).

Since the spaces of specific structured matrices above will be used throughout this paper, we denote each space as follows:
\begin{itemize}
        \item [(\romannumeral1)] $\mathcal{S}_n$ : the affine variety of $n \times n$ symmetric matrices;
        \item [(\romannumeral2)] $\mathcal{T}_n$ : the affine variety of $n \times n$ Toeplitz matrices;
        \item [(\romannumeral3)] $\mathcal{H}_n$ : the affine variety of $n \times n$ Hankel matrices;
        \item [(\romannumeral4)] $\mathcal{UBD}_n$ (resp. $\mathcal{LBD}_n$) : the affine variety of $n \times n$ upper (resp. lower) bidiagonal matrices;
        \item [(\romannumeral5)] $\mathcal{TD}_n$ : the affine variety of $n \times n$ tridiagonal matrices;
        \item [(\romannumeral6)] $\Lambda_n$ : the affine variety of $n \times n$ skew-symmetric matrices;
        \item [(\romannumeral7)] $\mathcal{C}_n$ : the affine variety of $n \times n$ companion matrices.
\end{itemize}

\section{Structured Matrix Factorization Length}\label{section:Structured Matrix Factorization Length}

In this section, we introduce the main objects of study in this paper, and show that the definitions formulated in terms of the Zariski topology admit equivalent formulations using the Euclidean topology. We then examine several examples.

\subsection{Definitions}

We emphasize that, for each structure mentioned in Theorem \ref{thm:previous results}, the space of the structured matrices is an affine variety. This observation makes it natural to define the notion of \emph{structured matrix factorization length} as follows:

\begin{definition}\label{def:structured factorization length}
    Let $A \in \C^{n \times n}$, and let $X$ be an affine variety in $\mathbb{C}^{n \times n}$. 
    \begin{itemize}
        \item [(\romannumeral1)] If $A$ can be factorized as $A=M_1 \cdots M_r$ for some $M_1,...,M_r \in X$, then the $r$-tuple $(M_1,...,M_r) \in X^{\times r}$ is called an  \emph{$X$-factorization} of $A$ of \emph{length $r$}.
        \item [(\romannumeral2)] The \emph{$X$-factorization length} of $A$, denoted by $\mathbf{FL}_X(A)$, is defined by
        \begin{equation*}
            \mathbf{FL}_{X}(A)=\min\{r~|~A=M_1 \cdots M_r~\text{for some}~ M_1,...,M_r \in X\},
        \end{equation*}
        that is, the smallest $r$ such that $A$ has an $X$-factorization of length $r$. If there is no $X$-factorization of $A$, then we let $\mathbf{FL}_X(A)=\infty$.
    \end{itemize}
\end{definition}

For an affine variety $X$ in $\mathbb{C}^{n \times n}$, let $X^{\times r}$ simply denote the $r$-product of $X$, and let
    \begin{equation*}
         m_r: X^{\times r} \rightarrow \C^{n \times n},~(M_1,...,M_r) \mapsto M_1 \cdots M_r
    \end{equation*}
be the matrix multiplication map. Consider the image of the map $m_r$,
\begin{equation}\label{eq:alternative definition of structured matrix factorization length}
    \mu_r^0(X):=m_r(X^{\times r}).
\end{equation}
For $A \in \mathbb{C}^{n \times n}$, it is obvious that
\begin{equation*}
    \mathbf{FL}_X(A)=\min\{r~|~A \in \mu_r^0(X)\}.
\end{equation*}

\begin{remark}\label{rmk:whether contains the identity matrix}
    If the variety $X$ contains the identity matrix $I_n \in \mathbb{C}^{n \times n}$ as in the case of $\mathcal{T}_n$, $\mathcal{LBD}_n$ and $\mathcal{UBD}_n$, then 
    \begin{equation}\label{eq:Stratification under Inclusions}
        \mu_r^0(X) \subseteq \mu_{r+1}^0(X)
    \end{equation}
    for each $r \in \mathbb{Z}_{\geq 1}$, because $A=M_1\cdots M_r$ (where $M_1,...,M_r \in X$) implies $A=M_1\cdots M_r \cdot I_n$. Hence, in this case, we obtain
    \begin{equation}\label{eq:Description of the Variety via SMFL}
        \mu_r^0(X)=\{A \in \mathbb{C}^{n \times n}~|~\mathbf{FL}_X(A) \leq r\}.
    \end{equation}
    On the other hand, if $X$ does not contain the identity matrix $I_n$ as in the case of $\mathcal{H}_n$, $\Lambda_n$, and $\mathcal{C}_n$, then the inclusion at (\ref{eq:Stratification under Inclusions}) and the equality at (\ref{eq:Description of the Variety via SMFL}) need not hold (see Examples \ref{ex:Hankel} and \ref{ex:Skew-symmetric Structure} below). Although $X$ does not contain $I_n$, whenever $\mu_k^0(X)$ contains $I_n$ for some $k$, we obtain the inclusion
    \begin{equation*}
        \mu_s^0(X) \subseteq \mu_{s+k}^0(X)
    \end{equation*}
    for all $s \geq 1$, because $A=M_1\cdots M_s$ and $I_n=M_1' \cdots M_k'$ imply $A=M_1\cdots M_s \cdot I_n=M_1\cdots M_s \cdot M_1' \cdots M_k'$.
\end{remark}

Since $\mu_r^0(X)$ is not Zariski closed in general (see Section \ref{subsection:Examples}), it is natural to replace $\mu_r^0(X)$ on (\ref{eq:alternative definition of structured matrix factorization length}) by its Zariski closure 
\begin{equation*}
    \mu_r(X):=\overline{\mu_r^0(X)}^{\text{Zar}}.
\end{equation*}
We will call it the \emph{$r$-th $X$-factorization variety}. Now, we obtain the following natural definition:

\begin{definition}\label{def:border structured factorization length}
    Let $A \in \C^{n \times n}$, and let $X $ be an affine variety in $\mathbb{C}^{n \times n}$. The \emph{border $X$-factorization length} of $A$, denoted by $\underline{\mathbf{FL}}_X(A)$, is defined as 
    \begin{equation*}
        \underline{\mathbf{FL}}_X(A)=\min\{r~|~A \in \mu_r(X)\}.
    \end{equation*}
    If $A \notin \mu_r(X)$ for any positive integer $r$, then we let $\underline{\mathbf{FL}}_X(A)=\infty$.
\end{definition}

\begin{remark}\label{rmk:border version of whether contains the identity matrix}
    As in Remark \ref{rmk:whether contains the identity matrix}, if the variety $X$ contains the identity matrix $I_n \in \mathbb{C}^{n \times n}$, then 
    \begin{equation}\label{eq:Border version of Stratification under Inclusions}
        \mu_r(X) \subseteq \mu_{r+1}(X)
    \end{equation}
    for each $r \in \mathbb{Z}_{\geq 1}$, and hence we obtain
    \begin{equation}\label{eq:Description of the Variety via border SMFL}
        \mu_r(X)=\{A \in \mathbb{C}^{n \times n}~|~\underline{\mathbf{FL}}_X(A) \leq r\}.
    \end{equation}
    On the other hand, if $X$ does not contain the identity matrix $I_n$, then the inclusion at (\ref{eq:Border version of Stratification under Inclusions}) and the equality at (\ref{eq:Description of the Variety via border SMFL}) need not hold. As in Remark \ref{rmk:whether contains the identity matrix}, whenever $\mu_k^0(X)$ contains $I_n$ for some $k$, we obtain the inclusion
    \begin{equation}\label{eq:Border version of general Stratification under Inclusions}
        \mu_s(X) \subseteq \mu_{s+k}(X)
    \end{equation}
    for all $s$, because $\mu_s^0(X) \subseteq \mu_{s+k}^0(X)$ holds. Furthermore, whenever $X$ is irreducible and $\mu_k(X)$ (instead of $\mu_k^0(X)$) contains $I_n$ for some $k$, we obtain the inclusion at (\ref{eq:Border version of general Stratification under Inclusions}) for all $s \geq 1$ again. We will prove it in Section \ref{subsection:Zariski closure versus Euclidean closure} (see Corollary \ref{cor:when the identity matrix is on the closure}).
\end{remark}

From the definitions, we obtain $\underline{\mathbf{FL}}_X(A) \leq \mathbf{FL}_X(A)$ for every affine variety $X \subseteq \mathbb{C}^{n \times n}$ and every $A \in \mathbb{C}^{n \times n}$. In addition, the inequality can be strict (see Section \ref{subsection:Examples}). Now, we introduce the notions which are directly related to Theorem \ref{thm:previous results}:

\begin{definition}
    Let $X$ be an affine variety in $\mathbb{C}^{n \times n}$.
    \begin{itemize}
        \item [(\romannumeral1)] The \emph{maximal $X$-factorization length}, denoted by $\mathbf{FL}_{\text{max}}(X)$, is defined by
    \begin{equation*}
        \mathbf{FL}_{\text{max}}(X)=\min\{r~|~\mu_r^0(X)=\mathbb{C}^{n \times n}\}.
    \end{equation*}
        \item [(\romannumeral2)] The \emph{generic $X$-factorization length}, denoted by $\mathbf{FL}_{\text{gen}}(X)$, is defined by
    \begin{equation*}
        \mathbf{FL}_{\text{gen}}(X)=\min\{r~|~\mu_r(X)=\mathbb{C}^{n \times n}\}.
    \end{equation*}
    \end{itemize}
    
\end{definition}

\subsection{Zariski closure versus Euclidean closure}\label{subsection:Zariski closure versus Euclidean closure}

As we mentioned in Section \ref{section:Preliminaries}, the Zariski closure and Euclidean closure of a set are different in general. In particular, the Zariski closure can be much larger than the Euclidean closure. In this section, we will show that the Zariski and Euclidean closure for the set $\mu_r^0(X)$ are the same if $X$ is irreducible (see Theorem \ref{thm:Zariski vs. Euclidean closures} below). The key lemma to show this property is as follows:

\begin{lemma}[see {\cite[Theorem 3.1.6.1]{landsberg2017geometry}}]\label{Lemma:Euclidean and Zariski closures}
    Let $Z \subset \mathbb{C}^n$ be a subset. If $Z$ contains a Zariski open subset of $\overline{Z}^{\text{Zar}}$ and $\overline{Z}^{\text{Zar}}$ is irreducible, then $\overline{Z}^{\text{Zar}}=\overline{Z}^{\text{Euc}}$.
\end{lemma}

In order to use this lemma, we need to show that $\overline{\mu_r(X)}^{\text{Zar}}$ is irreducible. The following proposition guarantees this property:

\begin{proposition}\label{prop:irreducibility}
    Let $X \subseteq \C^{n \times n}$ be an irreducible affine variety, and let $r \in \Z_{\geq 1}$. Then $\mu_r(X):=\overline{\mu_r^0(X)}^{\text{Zar}}$ is an irreducible affine variety in $\C^{n \times n}$.
\end{proposition}
\begin{proof}
Note that the map $m_r:X^{\times r} \rightarrow \mathbb{C}^{n \times n}$ is a polynomial map, and so a morphism between the affine varieties. Since $X$ is an irreducible affine variety (and so an irreducible constructible set), then $X^{\times r}$ is an irreducible constructible set. Hence, $m_r(X^{\times r})$ is an irreducible constructible set in $\C^{n \times n}$. Since taking the Zariski closure preserves the irreducibility (see \cite[Example 1.1.4]{hartshorne2013algebraic}), then the affine variety $\overline{m_r(X^{\times r})}^{\text{Zar}}$ is also irreducible.
\end{proof}

Note that the affine varieties of structured matrices in Theorem \ref{thm:previous results} are all irreducible, because the corresponding defining equations are linear equations. Hence, we obtain the following corollary:

\begin{corollary}
    The affine variety $\mu_r(X)$ is irreducible for every $r \in \mathbb{Z}_{\geq 1}$, if $X$ is $\mathcal{S}_n$, $\mathcal{T}_n$, $\mathcal{H}_n$, $\mathcal{UBD}_n$, $\mathcal{LBD}_n$, $\mathcal{TD}_n$, $\Lambda_n$, or $\mathcal{C}_n$.
\end{corollary}

In addition, we obtain the main theorem of this section as follows:

\begin{theorem}\label{thm:Zariski vs. Euclidean closures}
    Let $X \subseteq \C^{n \times n}$ be an irreducible affine variety, and let $r \in \Z_{\geq 1}$. Then
    \begin{equation*}
    \overline{\mu_r^0(X)}^{\text{Zar}}=\overline{\mu_r^0(X)}^{\text{Euc}}.
\end{equation*}
\end{theorem}
\begin{proof}
    The set $\mu_r^0(X)=m_r(X^{\times r})$ contains a dense open subset of $\mu_r(X)=\overline{m_r(X^{\times r})}^{\text{Zar}}$ by Chevalley's Theorem (see {\cite[Corollary AG.10.2]{borel2012linear}}). Thus, we obtain that $\overline{\mu_r^0(X)}^{\text{Zar}}=\overline{\mu_r^0(X)}^{\text{Euc}}$ by Proposition \ref{prop:irreducibility} and Lemma \ref{Lemma:Euclidean and Zariski closures}.
\end{proof}

Note that Lemma \ref{Lemma:Euclidean and Zariski closures} is proved in \cite{landsberg2017geometry} by showing that, for each point  $p \in \overline{Z}^{\text{Zar}}$, there exists an analytic curve $C(t)$ in Z which converges to the point $p$. Hence, we obtain the following corollary: 

\begin{corollary}\label{cor:limit expression}
    Let $A \in \mathbb{C}^{n \times n}$, and let $X \subseteq \mathbb{C}^{n \times n}$ be an irreducible affine variety. Then the $r$-th $X$-factorization variety can be expressed as
    \begin{equation*}
        \mu_r(X)=\{A=\lim_{\epsilon \rightarrow 0}M_1(\epsilon)\cdots M_r(\epsilon)~|~M_{i}(\epsilon) \in X~\text{for all $\epsilon \neq 0$},~\text{for each $i=1,...,r$}\}.
    \end{equation*}
    In particular, the border $X$-factorization length of $A$ can be expressed as
    \begin{equation*}
        \underline{\mathbf{FL}}_X(A)=\min\{r~|~A=\lim_{\epsilon \rightarrow 0}M_1(\epsilon)\cdots M_r(\epsilon),~\text{where}~M_{i}(\epsilon) \in X~\text{for all $\epsilon \neq 0$},~\text{for each $i=1,...,r$}\}.
    \end{equation*}
\end{corollary}

Now, we prove the last assertion in Remark \ref{rmk:border version of whether contains the identity matrix} by using Corollary \ref{cor:limit expression}.

\begin{corollary}\label{cor:when the identity matrix is on the closure}
    If $X$ is an irreducible affine variety and $\mu_r(X)$ contains the identity matrix $I_n$, then
    \begin{equation*}
        \mu_s(X) \subseteq \mu_{s+r}(X)
    \end{equation*}
    for all $s \geq1$.
\end{corollary}
\begin{proof}
    Let $s$ be a positive integer. By Corollary \ref{cor:limit expression}, an element $A \in \mu_s(X)$ can be expressed as $A=\lim_{\epsilon \rightarrow 0} M_1(\epsilon)\cdots M_s(\epsilon)$ where $M_1(\epsilon),...,M_s(\epsilon) \in X$ for $\epsilon \neq 0$. Since $I_n \in \mu_r(X)$, then it can be expressed as $I_n=\lim_{\epsilon \rightarrow 0} M_1'(\epsilon)\cdots M_r'(\epsilon)$ where $M_1'(\epsilon),...,M_r'(\epsilon) \in X$ for $\epsilon \neq 0$ by Corollary \ref{cor:limit expression}. Hence, we obtain that
    \begin{equation*}
        A=\lim_{\epsilon \rightarrow 0}M_1(\epsilon)\cdots M_s(\epsilon)=\lim_{\epsilon \rightarrow 0}M_1(\epsilon)\cdots M_s(\epsilon) \cdot I_n =\lim_{\epsilon \rightarrow 0}M_1(\epsilon)\cdots M_s(\epsilon)M_1'(\epsilon)\cdots M_r'(\epsilon),
    \end{equation*}
    and thus $A \in \mu_{s+r}(X)$ by Corollary \ref{cor:limit expression} again.
\end{proof}


\subsection{Examples}\label{subsection:Examples}

In this section, we discuss several examples. Throughout this section, we let $n \geq 2$.

\begin{example}
    Let $X$ be an affine variety in $\mathbb{C}^{n \times n}$. Obviously, every element $A \in X$ satisfies
    \begin{equation*}
        \mathbf{FL}_X(A)=\underline{\mathbf{FL}}_X(A)=1.
    \end{equation*}
\end{example}

\begin{example}[Symmetric matrix structure]
    Consider $\mathcal{S}_n$ the affine variety of $n \times n$ symmetric matrices. By Theorem \ref{thm:previous results}.(\romannumeral1), we obtain that  
    \begin{equation}\label{eq:factorization variety of symmetric matrix structure}
        \mu_2^0(\mathcal{S}_n)=\mu_2(\mathcal{S}_n)=\C^{n \times n}.
    \end{equation}
    Hence, all non-symmetric matrices are of (border) $\mathcal{S}_n$-factorization length $2$. Moreover, the equalities at (\ref{eq:factorization variety of symmetric matrix structure}) imply that
    \begin{equation*}
        \mathbf{FL}_{\text{max}}(\mathcal{S}_n)=\mathbf{FL}_{\text{gen}}(\mathcal{S}_n)=2.
    \end{equation*}
\end{example}

\begin{example}[Toeplitz matrix structure]\label{ex:Toeplitz}    Consider $\mathcal{T}_n$ the affine variety of $n \times n$ Toeplitz matrices. By Theorem \ref{thm:previous results}.(\romannumeral2), we obtain that
    \begin{equation}\label{eq:generic and maximal length of toeplitz}
        \mathbf{FL}_{\text{max}}(\mathcal{T}_n) \leq 2n+5~\text{ and }~\mathbf{FL}_{\text{gen}}(\mathcal{T}_n) \leq \lfloor n/2 \rfloor +1.
    \end{equation}
    Furthermore, it is also proved in \cite[Corollary 1]{ye2016every} that the bound $\lfloor n/2 \rfloor +1$ for generic length is sharp, and hence we have the equality
    \begin{equation*}
        \mathbf{FL}_{\text{gen}}(\mathcal{T}_n) = \lfloor n/2 \rfloor +1.
    \end{equation*}
    
    On the other hand, the precise value of the maximal $\mathcal{T}_n$-factorization length has not been determined yet. It was conjectured that it is also $\lfloor n/2 \rfloor +1$ in \cite[Conjecture 1]{ye2016every}, but it was shown in \cite{garcia2025minimum} that this conjecture is false by proving that the diagonal matrix 
    \begin{equation*}
        D=\begin{bmatrix}
            1 & 0 & 0\\
            0 & 2 & 0\\
            0 & 0 & 3
        \end{bmatrix}.
    \end{equation*}
    is of $\mathbf{FL}_{\mathcal{T}_3}(D) = 3$ whereas $\lfloor 3/2 \rfloor +1=2$. The proof idea is as follows. First, we can check that the system of polynomial equations derived from $D=T_1T_2$, where the $T_i$'s are Toeplitz matrices with independent variables, does not have any solution. This implies that $D$ cannot be expressed as a product of two Toeplitz matrices. Second, from the equality 
    \begin{equation*}
        D=\begin{bmatrix}
            0 & 0 & 1\\
            1 & 0 & 0\\
            0 & 1 & 0
        \end{bmatrix}\begin{bmatrix}
            0 & 0 & 2\\
            1 & 0 & 0\\
            0 & 1 & 0
        \end{bmatrix}\begin{bmatrix}
            0 & 0 & 3\\
            1 & 0 & 0\\
            0 & 1 & 0
        \end{bmatrix}
    \end{equation*} 
    we can say that $D$ is a product of three Toeplitz matrices. Note that this example shows that the $X$-factorization length and border $X$-factorization length do not coincide in general. Furthermore, from Corollary \ref{cor:limit expression}, the matrix $D$ must be expressed by a limit of products of two Toeplitz matrices. For example, the matrix $D$ has the following limit expression
    \begin{equation*}
           D=\lim_{\epsilon \rightarrow 0}\begin{bmatrix}
        1 & 0 & 1/\epsilon^2\\
        -\epsilon & 1 & 0\\
        -\epsilon^2 & -\epsilon & 1
           \end{bmatrix}\begin{bmatrix}
        1 & -1/\epsilon & -1/\epsilon^2\\
        \epsilon & 1 & -1/\epsilon\\
        0 & \epsilon & 1
        \end{bmatrix}.
    \end{equation*}
\end{example}

\begin{example}[Hankel matrix structure]\label{ex:Hankel}
    Consider $\mathcal{H}_n$ the affine variety of $n \times n$ Hankel matrices. By Theorem \ref{thm:previous results}.(\romannumeral3), we obtain that
    \begin{equation*}
        \mathbf{FL}_{\text{max}}(\mathcal{H}_n) \leq 2n+5~\text{ and }~\mathbf{FL}_{\text{gen}}(\mathcal{H}_n) \leq \lfloor n/2 \rfloor +1.
    \end{equation*}
    In \cite{ye2016every}, this result was derived from the inequalities at (\ref{eq:generic and maximal length of toeplitz}). The authors used the fact that the reversal permutation matrix $J$ described at (\ref{eq:reversal permutation matrix}) satisfies $J^2=I_n$, and the fact that both $TJ$ and $JT$ are Hankel matrices for a Toeplitz matrix $T$. Specifically, if a generic matrix $A$ has the expression $A=M_1\cdots M_r$ where each $M_i$ is a Toeplitz matrix, then we obtain $A=(M_1J)(JM_2)\cdots (M_{r-1}J)(JM_r)$ when $r$ is even and $JA=(JM_1)(M_2J)(JM_3)\cdots(M_{r-1}J)(JM_r)$ when $r$ is odd, and hence a generic matrix can be expressed as a product of $r$ Hankel matrices.

    Note that $I_n \notin \mathcal{H}_n$. Remark \ref{rmk:whether contains the identity matrix} says that $\mathcal{H}_n$ need not be included in $\mu_2(\mathcal{H}_n)$. Consider the $3 \times 3$ Hankel matrix
    \begin{equation*}
        H:=\begin{bmatrix}
            0 & 1 & 0 \\
            1 & 0 & 0\\
            0 & 0 &0
        \end{bmatrix}.
    \end{equation*}
    We show that it is not an element of $\mu_2^0(\mathcal{H}_3)$. Suppose that $H=AB$, where $A$ and $B$ are $3 \times 3$ Hankel matrices. Then 
    \begin{equation*}
        AB=H=H^T=B^TA^T=BA,
    \end{equation*}
    and so 
    \begin{equation*}
        AH=AAB=ABA=HA~\text{ and }~BH=BAB=ABB=HB.
    \end{equation*}
    From the conditions $AH=HA$ and $BH=HB$, we obtain that the Hankel matrices $A,B$ must be of the forms
    \begin{equation*}
        A=\begin{bmatrix}
            0 & a_{-1} & 0\\
            a_{-1} & 0 & 0\\
            0  & 0 & a_2
        \end{bmatrix}~\text{ and }~B=\begin{bmatrix}
            0 & b_{-1} & 0\\
            b_{-1} & 0 & 0\\
            0  & 0 & b_2
        \end{bmatrix}.
    \end{equation*}
    This implies that
    \begin{equation*}
        H=AB=\begin{bmatrix}
            a_{-1}b_{-1} & 0 & 0\\
            0 & a_{-1}b_{-1} & 0\\
            0 & 0 & a_2b_2
        \end{bmatrix},
    \end{equation*}
    a contradiction. Hence, $H \notin \mu_2^0(\mathcal{H}_3)$, and so 
    \begin{equation*}
        \mathcal{H}_3 \nsubseteq \mu_2^0(\mathcal{H}_3).
    \end{equation*}
    However, $\mathcal{H}_3 \subseteq \mu_2(\mathcal{H}_3)$, because $\mu_2(\mathcal{H}_3)=\mathbb{C}^{3 \times 3}$ by Theorem \ref{thm:previous results}.(\romannumeral3). We will show that $\mathcal{H}_6 \nsubseteq \mu_2(\mathcal{H}_6)$ in Section \ref{subsection:Defining equations via Displacement} (see Example \ref{ex:defining equations for Hankel matrix structure}). On the other hand, since $I_n=J^2 \in \mu_2^0(\mathcal{H}_n)$ where $J$ is the reversal permutation matrix, then we have that $\mu_s(\mathcal{H}_n) \subseteq \mu_{s+2}(\mathcal{H}_n)$ for all $s \geq 1$.
\end{example}

\begin{example}[Bidiagonal matrix structure]\label{ex:bidiagonal}
    Consider $\mathcal{UBD}_n$ the space of $n \times n$ upper bidiagonal matrices. Since every nonzero strictly lower triangular matrix $L$ satisfies
    \begin{equation*}
        \mathbf{FL}_X(L)=\infty,
    \end{equation*}
    we cannot talk about the maximal and generic $\mathcal{UBD}_n$-factorization length. However, if we restrict to the space of upper (resp. lower) triangular matrices, then we can talk about the analogues of them. Let $2 \leq k \leq n$. An $n \times n$ matrix $A=[a_{ij}]$ is called an upper  $k$-diagonal matrix if it is an upper triangular matrix and $a_{ij}=0$ if $j-i\geq k$.
    It already was proved that a generic $n \times n$ upper (resp. lower) $k$-diagonal matrix is a product of $k-1$ upper (resp. lower) bidiagonal matrices in \cite[Lemma 5.1]{ye2017new}. That is, $\mu_{k-1}(\mathcal{UBD}_n)$ (resp. $\mu_{k-1}(\mathcal{LBD}_n)$) is the affine variety of $n \times n$ upper (resp. lower) $k$-diagonal matrices. Hence, we can say that every $n \times n$ upper $k$-diagonal matrix $U_k$ is of $\underline{\mathbf{FL}}_{\mathcal{UBD}_n}(U_k) \leq k-1$. However, it does not hold in general that every $n \times n$ upper $k$-diagonal matrix $U_k$ is of $\mathbf{FL}_{\mathcal{UBD}_n}(U_k)\leq k-1$. For example, consider the $3$-diagonal matrix
    \begin{equation}\label{eq:example of upper triangular matrix}
        U_3:=\begin{bmatrix}
        1 & 0 & 1 & 0 & 0\\
        0 & 1 & 0 & 0 & 0\\
        0 & 0 & 1 & 0 & 1\\
        0 & 0 & 0 & 1 & 0\\
        0 & 0 & 0 & 0 & 1
        \end{bmatrix}.
    \end{equation}
    If we derive a system of polynomial equations from 
    \begin{equation*}
        U_3=\begin{bmatrix}
            x_{11} & x_{12} & 0 & 0 & 0\\
            0 & x_{22} & x_{23} & 0 & 0\\
            0 & 0 & x_{33} & x_{34} & 0\\
            0 & 0 & 0 & x_{44} & x_{45}\\
            0 & 0 & 0 & 0 & x_{55}
        \end{bmatrix}\begin{bmatrix}
            y_{11} & y_{12} & 0 & 0 & 0\\
            0 & y_{22} & y_{23} & 0 & 0\\
            0 & 0 & y_{33} & y_{34} & 0\\
            0 & 0 & 0 & y_{44} & y_{45}\\
            0 & 0 & 0 & 0 & y_{55}
        \end{bmatrix},
    \end{equation*}
    then we can check that the system has no solution, and so $U_3$ cannot be expressed as a product of two upper bidiagonal matrices, that is, $\mathbf{FL}_{\mathcal{UBD}_5}(U_3)\geq 3$. On the other hand, from Corollary \ref{cor:limit expression}, the matrix $U_3$ must be expressed by a limit of products of two upper bidiagonal matrices. For example, the matrix $U_3$ has the following limit expression
    \begin{equation*}
        U_3=\lim_{\epsilon \rightarrow 0} \begin{bmatrix}
            1 & 1 & 0 & 0 & 0\\
            0 & 1 & -1 & 0 & 0\\
            0 & 0 & 1 & -\epsilon & 0\\
            0 & 0 & 0 & 1 & \epsilon^{-1}\\
            0 & 0 & 0 & 0 & 1
        \end{bmatrix}
        \begin{bmatrix}
            1 & -1 & 0 & 0 & 0\\
            0 & 1 & 1 & 0 & 0\\
            0 & 0 & 1 & \epsilon & 0\\
            0 & 0 & 0 & 1 & -\epsilon^{-1}\\
            0 & 0 & 0 & 0 & 1
        \end{bmatrix}.
    \end{equation*}
\end{example}

\begin{example}[Tridiagonal matrix structure]
    Consider $\mathcal{TD}_n$ the affine variety of $n \times n$ tridiagonal matrices. By Theorem \ref{thm:previous results}.(\romannumeral5), we obtain that
    \begin{equation*}
        \mathbf{FL}_{\text{max}}(\mathcal{TD}_n)\leq 8(n-1)~\text{ and }~\mathbf{FL}_{\text{gen}}(\mathcal{TD}_n) \leq 2(n-1).
    \end{equation*}
    The upper bound for the generic length is proved by applying Theorem \ref{thm:previous results}.(\romannumeral4) to an LU factorization of a generic $n \times n$ matrix. However, the LU factorization is not an optimal way to get a $\mathcal{TD}_n$-factorization. For example, for the matrix $U_3$ at (\ref{eq:example of upper triangular matrix}), consider $U_3^TU_3$. This expression is already an LU factorization of the matrix. Since $\mathbf{FL}_{\mathcal{UBD}_5}(U_3)\geq 3$, then we only get a $\mathcal{TD}_5$-factorization of $U_3^TU_3$ of length $\geq 6$ in this way. However, it has the factorization
    \begin{equation*}
        U_3^TU_3=\begin{bmatrix}
            1 & 0 & 1 & 0 & 0\\
            0 & 1 & 0 & 0 & 0\\
            1 & 0 & 2 & 0 & 1\\
            0 & 0 & 0 & 1 & 0\\
            0 & 0 & 1 & 0 & 2
        \end{bmatrix}=\begin{bmatrix}
            0 & 1 & 0 & 0 & 0\\
            1 & 0 & 0 & 0 & 0\\
            0 & 1 & 1 & 0 & 0\\
            0 & 0 & 0 & 0 & 1\\
            0 & 0 & 0 & 1 & 0
        \end{bmatrix}\begin{bmatrix}
            0 & 1 & 0 & 0 & 0\\
            1 & 0 & 0 & 0 & 0\\
            0 & 0 & 1 & 0 & 0\\
            0 & 0 & 1 & 0 & 1\\
            0 & 0 & 0 & 1 & 0
        \end{bmatrix}\begin{bmatrix}
            0 & 1 & 0 & 0 & 0\\
            1 & 0 & 0 & 0 & 0\\
            0 & 0 & 1 & 1 & 0\\
            0 & 0 & 0 & 0 & 1\\
            0 & 0 & 0 & 1 & 0
        \end{bmatrix}\begin{bmatrix}
            0 & 1 & 0 & 0 & 0\\
            1 & 0 & 1 & 0 & 0\\
            0 & 0 & 1 & 0 & 0\\
            0 & 0 & 0 & 0 & 1\\
            0 & 0 & 0 & 1 & 0
        \end{bmatrix},
    \end{equation*}
    and hence $\mathbf{FL}_{\mathcal{TD}_5}(U_3^TU_3) \leq 4$.

    As in the bidiagonal case, there exists a matrix $S$ such that $\mathbf{FL}_{\mathcal{TD}_n}(S) \neq \underline{\mathbf{FL}}_{\mathcal{TD}_n}(S)$. For example, consider
    \begin{equation*}
        S:=\begin{bmatrix}
            0 & 0 & 1 & 0\\
            0 & 0 & 0 & 1\\
            1 & 0 & 0 & 0\\
            0 & 1 & 0 & 0
        \end{bmatrix}.
    \end{equation*}
    If we derive a system of polynomial equations from
    \begin{equation*}
        S=\begin{bmatrix}
            x_{11} & x_{12} & 0 & 0\\
            x_{21} & x_{22} & x_{23} & 0\\
            0 & x_{32} & x_{33} & x_{34}\\
            0 &0 & x_{43} & x_{44}
        \end{bmatrix}
        \begin{bmatrix}
            y_{11} & y_{12} & 0 & 0\\
            y_{21} & y_{22} & y_{23} & 0\\
            0 & y_{32} & y_{33} & y_{34}\\
            0 &0 & y_{43} & y_{44}
        \end{bmatrix},
    \end{equation*}
    then we can check that this system has no solution, and so $S$ cannot be expressed as a product of two tridiagonal matrices, that is, $\mathbf{FL}_{\mathcal{TD}_4}(S) \geq 3$. However, since the matrix $S$ has the limit expression
    \begin{equation*}
        S=\lim_{\epsilon \rightarrow 0}\begin{bmatrix}
          1 & \epsilon^{-1} & 0 & 0 \\
          0 & 0 & \epsilon & 0\\
          0 & \epsilon & 0 & 0\\
          0 & 0 & \epsilon^{-1} & -\epsilon^{-2}
        \end{bmatrix}
        \begin{bmatrix}
         -\epsilon^{-2} & 0 & 0 & 0\\
         \epsilon^{-1} & 0 & \epsilon & 0\\
         0 & \epsilon & 0 & \epsilon^{-1}\\
         0 & 0 & 0 & 1
        \end{bmatrix},
    \end{equation*}
    then we obtain that $\underline{\mathbf{FL}}_{\mathcal{TD}_4}(S)=2$.
\end{example}

\begin{example}\label{ex:Skew-symmetric Structure}
    Consider $\Lambda_n$ the affine variety of $n \times n$ skew-symmetric matrices. By Theorem \ref{thm:previous results}.(\romannumeral6), we obtain that if $n \geq 4$, then
    \begin{equation*}
        \mathbf{FL}_{\text{max}}(\Lambda_{2n}) \leq 13~\text{ and }~\mathbf{FL}_{\text{gen}}(\Lambda_{2n}) \leq 3.
    \end{equation*}
    For the cases of $n=2$ and $n=3$, we also can obtain analogous statements from Theorem \ref{thm:previous results}.(\romannumeral6). Note that, for any invertible $(2n+1) \times (2n+1)$ matrix $A$, we obtain
    \begin{equation*}
        \mathbf{FL}_{\Lambda_{2n+1}}(A)=\underline{\mathbf{FL}}_{\Lambda_{2n+1}}(A)=\infty,
    \end{equation*}
    because any $(2n+1) \times (2n+1)$ skew-symmetric matrix is singular. This is why Theorem \ref{thm:previous results}.(\romannumeral6) for $(2n+1) \times (2n+1)$ case was stated only for singular matrices.
    
    Note that $I_n \notin \Lambda_n$. Remark \ref{rmk:whether contains the identity matrix} says that the inclusion $\Lambda_n \subset \mu_2(\Lambda_n)$ need not hold. Consider the $4 \times 4$ skew-symmetric matrix 
    \begin{equation*}
        K:=\begin{bmatrix}
            0 & 1 & 0 & 0\\
            -1 & 0 & 0 & 0\\
            0 & 0 & 0 & 2\\
            0 & 0 & -2 & 0
        \end{bmatrix}.
    \end{equation*}
    We show that it is not an element of $\mu_2^0(\Lambda_4)$. Suppose that $K=AB$, where $A$ and $B$ are $4 \times 4$ skew-symmetric matrices. Since $K$ is invertible, then both $A$ and $B$ are also invertible. Then 
    \begin{equation}\label{eq:skew-symmetric condition for a contradiction}
        A^{-1}K=B=-B^T=-(A^{-1}K)^T=-K^T(A^{-1})^T=-(-K)(-A^{-1})=-KA^{-1}.
    \end{equation}
    If we let
    \begin{equation*}
        A^{-1}=\begin{bmatrix}
            0 & p & a & b\\
            -p & 0 & c & d\\
            -a & -c &0 & q\\
            -b & -d &-q & 0
        \end{bmatrix},
    \end{equation*}
    then the equality $A^{-1}K=-KA^{-1}$ from (\ref{eq:skew-symmetric condition for a contradiction}) implies that $p=q=a=b=c=d=0$, a contradiction. Hence, we really obtain that
    \begin{equation*}
        \Lambda_4 \nsubseteq \mu_2^0(\Lambda_4).
    \end{equation*}
    Furthermore, $\Lambda_4 \nsubseteq \mu_2(\Lambda_4)$, and we will show it in Section \ref{subsection:Defining equations via Displacement} (see Example \ref{ex:defining equations for skew-symmetric matrix structure}). On the other hand, every $2n \times 2n$ identity matrix $I_{2n}$ is a product of two skew-symmetric matrices, because
     \begin{equation*}
         I_{2n}=\begin{bmatrix}
             0 & I_n\\
             -I_n & 0
         \end{bmatrix}\begin{bmatrix}
             0 & -I_n\\
             I_n & 0
         \end{bmatrix}.
     \end{equation*}
     Since $I_{2n} \in \mu_2^0(\Lambda_{2n})$, then $\mu_s(\Lambda_{2n}) \subseteq \mu_{s+2}(\Lambda_{2n})$ for all $s \geq 1$.
\end{example}

\begin{example}[Companion matrix structure]\label{ex:companion matrix structure}
    Consider $\mathcal{C}_n$ the affine variety of $n \times n$ companion matrices. By Theorem \ref{thm:previous results}.(\romannumeral7), we obtain that $\mathbf{FL}_{\text{gen}}(\mathcal{C}_n) \leq n$. Furthermore, if $1 \leq k \leq n-1$, then a product of $k$ $n\times n$ companion matrices has the $(1,1)$-entry as $0$. Thus, we obtain the equality
    \begin{equation*}
        \mathbf{FL}_{\text{gen}}(\mathcal{C}_n) = n.
    \end{equation*}
    From Theorem \ref{thm:previous results}.(\romannumeral7) again, we have that every $n \times n$ matrix is a product of $4n$ companion matrices and a diagonal matrix. Since a diagonal matrix is not a companion matrix, an additional argument is required to obtain an upper bound of $\mathbf{FL}_{\text{max}}(\mathcal{C}_n)$. Note that an $n \times n$ diagonal matrix can be expressed as a product of $n$ companion matrices, because the following equality holds:
    \begin{equation*}
        \begin{bmatrix}
            d_1 & & & \\
            & d_2 & & \\
             & & \ddots & \\
             & & & d_n
        \end{bmatrix}=\begin{bmatrix}
            0 &\cdots & 0 & d_1\\
            1 &  &  & 0\\
             & \ddots &  & \vdots \\
             & &  1 & 0
        \end{bmatrix}\begin{bmatrix}
            0 &\cdots & 0 & d_2\\
            1 &  &  & 0\\
             & \ddots &  & \vdots \\
             & &  1 & 0
        \end{bmatrix}\cdots \begin{bmatrix}
            0 &\cdots & 0 & d_n\\
            1 &  &  & 0\\
             & \ddots &  & \vdots \\
             & &  1 & 0
        \end{bmatrix}.
    \end{equation*}
    Thus, we have that
    \begin{equation*}
        \mathbf{FL}_{\text{max}}(\mathcal{C}_n) \leq 5n.
    \end{equation*}

    In \cite[Section 6.1]{ye2017new}, it was shown that an $n \times n$ matrix $A=[a_{ij}]$ with $a_{11}=0 $ and $a_{12}=1$ is not a product of $n$ companion matrices. For example, consider the matrix
    \begin{equation*}
        A:=[a_{ij}]=\begin{bmatrix}
            0 & 1 & 0\\
            0 & 0 & 0\\
            0 & 0 & 0
        \end{bmatrix},
    \end{equation*}
    which is not in $\mu_3^0(\mathcal{C}_3)$. By Theorem \ref{thm:previous results}.(\romannumeral7), we have that $A \in \mu_3(\mathcal{C}_3)$. From Corollary \ref{cor:limit expression}, the matrix $A$ must be expressed by a limit of products of three companion matrices. For example, $A$ has the following limit expression
    \begin{equation*}
        A=\lim_{\epsilon \rightarrow 0}\begin{bmatrix}
            0 & 0 & \epsilon\\
            1 & 0 & 0\\
            0 & 1 & 0
        \end{bmatrix}\begin{bmatrix}
            0 & 0 & 0\\
            1 & 0 & 0\\
            0 & 1 &1/\epsilon
        \end{bmatrix}\begin{bmatrix}
            0 & 0 & 0\\
            1 & 0 & 0\\
            0 & 1 & 0
        \end{bmatrix}.
    \end{equation*}
    In addition, we can notice that $A \notin \mu_2(\mathcal{C}_3)$, since $a_{31}=0 \neq 1$ whereas every matrix in $\mu_2(\mathcal{C}_3)$ has the $(3,1)$-entry as $1$. Although this does not imply $A \notin \mathcal{C}_3$ because $\mathcal{C}_3$ does not contain the identity matrix (see Remark \ref{rmk:whether contains the identity matrix}), the matrix $A$ is certainly not a companion matrix. Thus, we have that
    \begin{equation*}
        \underline{\mathbf{FL}}_{\mathcal{C}_3}(A)=3.
    \end{equation*}
\end{example}

\section{Geometry of Structured Matrix Factorizations}\label{section:Geometry of Structured Matrix Factorizations}

Once a new variety is introduced, it is natural to study its basic invariants, such as its dimension and degree. If the variety is not introduced by the defining equations, then it is also natural to study its defining equations. Hence, we propose the following problem as a geometric problem associated with structured matrix factorization length:

\begin{problem}\label{problem:main problems on the geometry of structured matrix factorization length}
    When an affine variety $X$ in $\mathbb{C}^{n \times n}$ and $r \in \mathbb{Z}_{\geq 1}$ are given, determine
    \begin{itemize}
        \item [(\romannumeral1)] $\dim \mu_r(X)$,
        \item [(\romannumeral2)] $\deg \mu_r(X)$, and
        \item [(\romannumeral3)] all defining equations of $\mu_r(X)$.
    \end{itemize}
\end{problem}

In this section, we suggest the methods to solve this problem, and solve it for some $X$ and $r$. In particular, we will devote substantial attention to the Toeplitz and Hankel case, i.e., the case where $X=\mathcal{T}_n$ or $X=\mathcal{H}_n$. 

\subsection{Dimension} \label{subsection:Dimension}

\subsubsection{General Dimension Bounds}

When irreducible affine varieties $X$ and $Y$ are given in $\mathbb{C}^{n \times n}$, the following lemma has been used to predict a number $r$ such that the matrix multiplication map $m_r:X^{\times r} \rightarrow Y$ is dominant, as in Proposition \ref{prop:lower bound of generic length}:

\begin{lemma}[see {\cite[Theorem 1.25]{shafarevich1994basic} and \cite[Proposition 3.2]{ye2017new}}]\label{lemma:fiber dimension}
    Assume that $Y \rightarrow Z$ is a dominant morphism between irreducible affine varieties, then $\dim Y = \dim Z + \dim f^{-1}(z)$ for a generic point $z \in Z$.
\end{lemma}

Here, our target is $\dim \mu_r(X)$, and so we will use this lemma for the case where $Y=X^{\times r}$ and $Z=\mu_r(X)$. In practice, it is hard to determine the exact value of the dimension of a generic fiber, but it is sometimes easy to get a positive lower bound $l$ of the dimension of a generic fiber, in particular, when the affine variety $X$ is a linear subspace of $\mathbb{C}^{n \times n}$. Then we obtain the upper bound of $\dim \mu_r(X)$:
\begin{equation*}
        \dim \mu_r(X) \leq r \cdot \dim X- l.
\end{equation*}

In addition, we will use the following lemma when improving the lower bound of $\dim \mu_r(X)$:

\begin{lemma}\label{lemma:the rank of a differential gives a lower bound of dimension}
    Let $\varphi:Y \rightarrow Z$ be a morphism between affine varieties, and let $y \in Y$ be a smooth point. Then
    \begin{equation*}
        \rank (d\varphi_y) \leq \dim \overline{\varphi(Y)}^{\text{Zar}}.
    \end{equation*}
\end{lemma}
\begin{proof}
    In this proof, we write the Zariski closure of a variety $X$ as $\overline{X}$ instead of $\overline{X}^{\text{Zar}}$ for convenience. Let $r=\rank (d\varphi_y)$, and let $Z$ be an affine variety in $\mathbb{C}^N$. Since $d\varphi_y:T_yY \rightarrow T_{\varphi(y)}Z$ has rank $r$ and $Z \subseteq \mathbb{C}^{N}$, then there exists a linear map $\pi:\mathbb{C}^N \rightarrow \mathbb{C}^r$ such that $d(\pi \circ \varphi)_y:T_yY \rightarrow T_{\pi(\varphi(y))}\mathbb{C}^r$ is surjective. Since $y \in Y$ is a smooth point, then there exists a Zariski open neighborhood $U \subseteq Y$ of $y$ such that $d(\pi \circ \varphi)_{y'}$ has rank $r$ for every $y' \in U$. Hence, $\dim \overline{\pi(\varphi(U))}= r$. Thus, we obtain that
    \begin{equation*}
        r=\dim \overline{\pi(\varphi(U))} \leq  \dim \overline{\pi\left(\overline{\varphi(Y)}\right)} \leq \dim \overline{\varphi(Y)},
    \end{equation*}
    where the last inequality is from Lemma \ref{lemma:fiber dimension}.
\end{proof}

\subsubsection{Toeplitz and Hankel matrix structures}

We will determine the dimensions of $\mu_r(\mathcal{T}_n)$ and $\mu_r(\mathcal{H}_n)$ for every $r,n \geq 1$ in this section (see Theorem \ref{thm:dimension for Toeplitz and Hankel structures}). In Example \ref{ex:Hankel}, we explained the argument in \cite{ye2016every} for showing that $\mathbf{FL}_{\text{gen}}(\mathcal{H}_n)=\lfloor n/2 \rfloor+1$ from the result that  $\mathbf{FL}_{\text{gen}}(\mathcal{T}_n)=\lfloor n/2 \rfloor+1$. The key ingredient was that the left (resp. right) multiplication by $J$ gives an isomorphism between $\mathcal{T}_n$ and $\mathcal{H}_n$. Since we will deal with the varieties $\mu_r(\mathcal{T}_n)$ and $\mu_r(\mathcal{H}_n)$ for various $r$, we will use similar arguments to show that
    \begin{equation*}
        \mu_r(\mathcal{T}_n)\cong\mu_r(\mathcal{H}_n)
    \end{equation*}
for arbitrary $r$. Then we will calculate $\dim \mu_r(\mathcal{T}_n)$.

\begin{proposition}\label{prop:isomorphism between Toeplitz and Hankel factorization varieties}
    The variety $\mu_r(\mathcal{T}_n)$ is isomorphic to $\mu_r(\mathcal{H}_n)$ for every $r,n \geq 1$.
\end{proposition}

\begin{proof}
    Consider the reversal permutation matrix $J$ described at (\ref{eq:reversal permutation matrix}), which satisfies that $J^2$ is the identity matrix $I_n$. Assume that $r$ is even. Then $\mu_r^0(\mathcal{T}_n)=\mu_r^0(\mathcal{H}_n)$, since if one contains $A=M_1\cdots M_r$ then another also contains $A=(M_1J)(JM_2)\cdots (M_{r-1}J)(JM_r)$. Hence, in this case, $\mu_r(\mathcal{T}_n)=\mu_r(\mathcal{H}_n)$. Now, assume that $r$ is odd. If $A=\lim_{\epsilon \rightarrow 0}M_1(\epsilon)\cdots M_r(\epsilon)$, then 
    \begin{equation*}
        JA=\lim_{\epsilon \rightarrow 0}(J\cdot M_1(\epsilon))(M_2(\epsilon)\cdot J)(J\cdot M_3(\epsilon))\cdots(M_{r-1}(\epsilon)\cdot J)(J\cdot M_r(\epsilon)).
    \end{equation*}
    Hence, the map $\varphi:\mu_r(\mathcal{T}_n) \rightarrow \mu_r(\mathcal{H}_n), ~M \rightarrow JM$ is well-defined. Moreover, since $J$ satisfies $J^2=I_n$, then the map $\varphi$ is an isomorphism.
\end{proof}

\begin{lemma}\label{lemma:upper bound of the dimension, toeplitz}
    The dimension of $\mu_r(\mathcal{T}_n)$ (and so the dimension of $\mu_r(\mathcal{H}_n)$) has the upper bound
    \begin{equation*}
        \dim \mu_r(\mathcal{T}_n) = \dim \mu_r(\mathcal{H}_n) \leq \min\{n^2, 2r(n-1)+1\}.
    \end{equation*}
\end{lemma}
\begin{proof}
    From Proposition \ref{prop:isomorphism between Toeplitz and Hankel factorization varieties}, $\dim \mu_r(\mathcal{T}_n)=\dim \mu_r(\mathcal{H}_n)$. Now, we prove the inequality for $\dim \mu_r(\mathcal{T}_n)$. As in the proof of Theorem \ref{thm:Zariski vs. Euclidean closures}, we have that the set $\mu_r^0(\mathcal{T}_n)$ contains a dense open subset $U$ of $\mu_r(\mathcal{T}_n)$. Assume that $A \in U \subseteq \mu_r^0(\mathcal{T}_n)$. Then $A$ can be expressed as $A=M_1\cdots M_r$ for some $M_1,...,M_r \in \mathcal{T}_n$. Hence, we also have that
    \begin{equation*}
        A=(s_1M_1)\cdots (s_rM_r)
    \end{equation*}
    for arbitrary $(s_1,...,s_r) \in \mathbb{C}^{\times r}$ such that $s_1\cdots s_r=1$. Furthermore, since $\mathcal{T}_n$ is a linear subspace of $\mathbb{C}^{n \times n}$, then each $s_iM_i$ is in $\mathcal{T}_n$ for every $s_i \in \mathbb{C}$. Since $\dim \{(s_1,...,s_r)\in \mathbb{C}^{\times r}~|~s_1\cdots s_r=1\}=r-1$, then $\dim m_r^{-1}(A) \geq r-1$. Therefore, $\dim \mu_r(\mathcal{T}_n) \leq r \cdot (2n-1)-(r-1)=2r(n-1)+1$ by Lemma \ref{lemma:fiber dimension}. Since $\dim \mu_r(\mathcal{T}_n)$ cannot exceed the dimension of the ambient space $ \mathbb{C}^{n \times n}$, then $\dim \mu_r(\mathcal{T}_n) \leq \min\{n^2, 2r(n-1)+1\}$.
\end{proof}

We now show that the inequality in this lemma is in fact an equality (see Theorem \ref{thm:dimension for Toeplitz and Hankel structures}). We will prove the equality by improving the lower bound of $\dim \mu_r(\mathcal{T}_n)$ up to $2r(n-1)+1$ when $\mu_r(\mathcal{T}_n) \subsetneq \mathbb{C}^{n \times n}$ by using Lemma \ref{lemma:the rank of a differential gives a lower bound of dimension}. The process of the proof is similar to that of Theorem \ref{thm:previous results}.(\romannumeral2) in \cite{ye2016every}. Let
\begin{equation}\label{eq:basis of the space of Toeplitz matrices}
    B_k= [b_{ij}],~\text{ where }~ b_{ij}=\begin{cases}
        1 & \text{if }~j-i=k\\
        0 & \text{otherwise}
    \end{cases}
\end{equation}
so that $\{B_{-n+1},...,B_{-1},B_0,B_1,\cdots,B_{n-1}\}$ is a basis of the vector space $\mathcal{T}_n$. The authors proved that the matrix multiplication map $m_r:\mathcal{T}_n^{\times (\lfloor n/2 \rfloor +1)} \rightarrow \mathbb{C}^{n \times n}$ is dominant  by showing that the differential $d(m_{\lfloor n/2 \rfloor +1})_{\tau}:\mathcal{T}_n^{\times (\lfloor n/2 \rfloor +1)} \rightarrow \mathbb{C}^{n \times n}$ is of full rank at $\tau=(T_{n-\lfloor n/2 \rfloor -1},T_{n-\lfloor n/2 \rfloor },...,T_{n-1})$, where
\begin{equation}\label{eq:special points of pencils of Toeplitz matrices}
    T_{n-i}=B_0+t_{n-i}(B_{n-i}-B_{-(n-i)}).
\end{equation}
We let $\tau$ be called Ye-Lim's special point. Choosing the special point 
\begin{equation}\label{eq:the special point for differential}
    \tau_{r}:=(T_{n-r},...,T_{n-1})
\end{equation}
by selecting last $r$ entries of $\tau$, we will additionally show that
\begin{equation}\label{eq:lower bound of the rank of differential with respect to Toeplitz matrix structure}
    \rank (d(m_r)_{\tau_r}) \geq 2r(n-1)+1 
\end{equation}
for $r=2,3,...,\lfloor n/2 \rfloor-1$. In order to illustrate the proof idea, we calculate the dimension of $\mu_r(\mathcal{T}_n)$ for some small $n$ and $r$.

\begin{example}
    Consider the case where $(n,r)=(4,2)$. We calculate the dimension of $\mu_2(\mathcal{T}_4) \subsetneq \mathbb{C}^{4 \times 4}$. Consider the matrix multiplication map $m_2:\mathcal{T}_4 \times \mathcal{T}_4 \rightarrow \mathbb{C}^{4 \times 4}$. From (\ref{eq:special points of pencils of Toeplitz matrices}), we obtain
    \begin{equation*}
        T_2=\begin{bmatrix}
            1 & 0 & 0 & 0\\
            0 & 1 & 0 & 0\\
            0 & 0 & 1 & 0\\
            0 & 0 & 0 & 1
        \end{bmatrix}+t_2\begin{bmatrix}
            0 & 0 & 1 & 0\\
            0 & 0 & 0 & 1\\
            -1 & 0 & 0 & 0\\
            0 & -1 & 0 & 0 
        \end{bmatrix}~\text{ and }~ T_3=\begin{bmatrix}
            1 & 0 & 0 & 0\\
            0 & 1 & 0 & 0\\
            0 & 0 & 1 & 0\\
            0 & 0 & 0 & 1
        \end{bmatrix}+t_3\begin{bmatrix}
            0 & 0 & 0 & 1\\
            0 & 0 & 0 & 0\\
            0 & 0 & 0 & 0\\
            -1 & 0 & 0 & 0 
        \end{bmatrix}.
    \end{equation*}
    From (\ref{eq:the special point for differential}), we obtain $\tau_2=(T_2,T_3)$. Let $X_2$ and $X_3$ be matrices such that $[X_2]_{i,j}:=x_{2,j-i}$ and $[X_{3}]_{i,j}:=x_{3,j-i}$ where the variables $x_{k,j-i}$'s are all independent. Then we have that
    \begin{equation*}
        d(m_2)_{\tau_2}(X_2,X_3)=X_2T_3+T_2X_3.
    \end{equation*}
    Consider the change of coordinates defined by
    \begin{equation}\label{eq:change of coordinates 1}
        \begin{bmatrix}
            y_j\\
            z_j
        \end{bmatrix}=\begin{bmatrix}
            1 & 1 \\
            1 & 0
        \end{bmatrix}\begin{bmatrix}
            x_{2,j}\\ x_{3,j}
        \end{bmatrix},~\text{ i.e., }~\begin{cases}
            y_j=x_{2,j}+x_{3,j}\\
            z_j=x_{2,j}.
        \end{cases}
    \end{equation}
    Then 
    \begin{equation*}
    \begin{aligned}
        &d(m_{2})_{\tau_2}(X_2,X_3)\\
        &=\begin{bmatrix}
            y_0 & y_1 & y_2 & y_3\\
            y_{-1} & y_0 & y_1 & y_2\\
            y_{-2} & y_{-1} & y_0 & y_1\\
            y_{-3} & y_{-2} & y_{-1} & y_0
        \end{bmatrix}+t_3\begin{bmatrix}
            -z_3 & 0 & 0 & z_0\\
            -z_2 & 0 & 0 & z_{-1}\\
            -z_1 & 0 & 0 & z_{-2}\\
            -z_0 & 0 & 0 & z_{-3}
        \end{bmatrix}+t_2\begin{bmatrix}
            y_{-2}-z_{-2} & y_{-1}-z_{-1} & y_{0}-z_{0} & y_{1}-z_{1}\\
            y_{-3}-z_{-3} & y_{-2}-z_{-2} & y_{-1}-z_{-1} & y_{0}-z_{0}\\
            -y_0+z_0 & -y_1+z_1 & -y_2+z_2 & -y_3+z_3\\
            -y_{-1}+z_{-1} & -y_0+z_0 & -y_1+z_1  & -y_2+z_2
        \end{bmatrix}\\
        &=:L.
    \end{aligned}
    \end{equation*}
    If we select the positions
    \begin{equation}\label{eq:picking Lij 1}
        \{L_{1,1}, L_{2,1}, L_{3,1}, L_{2,4}, L_{3,4}, L_{4,4}, L_{4,1}, L_{4,2}, L_{3,2}, L_{2,2}, L_{1,2}, L_{1,3}, L_{1,4}\}
    \end{equation}
    as in Figure \ref{fig:picking Lij 1} and the variables
    \begin{equation*}
        \{z_3,z_2,z_1,z_{-1},z_{-2},z_{-3},y_{-3},y_{-2},y_{-1},y_0,y_1,y_2,y_3\},
    \end{equation*}
    then, with respect to these choices of column and row indices, we obtain the following submatrix $M$ of a matrix representation of $(dm_2)_{\tau_2}$ :
    \begin{equation*}
        M=\begin{bNiceArray}{cccccc|ccccccc}[first-row,first-col]
        & z_3 & z_2 & z_1 & z_{-1} & z_{-2} & z_{-3} & y_{-3} & y_{-2} & y_{-1} & y_0 & y_1 & y_2 & y_3\\
        L_{1,1} & -t_3 & & & & -t_2 & & & t_2 & & 1 & & & \\
        L_{2,1} & & -t_3 & & & & -t_2 & t_2 &  & 1 & & & & \\
        L_{3,1} & & & -t_3 & & & & & 1 & & -t_2 & & & \\
        L_{2,4} & & & & t_3 & & & & & & t_2 & & 1 & \\
        L_{3,4} & t_2 & & & & t_3 & & & & & & 1 & & -t_2\\
        L_{4,4} & & t_2 & & & & t_3 & & & & 1 & & -t_2 & \\ \hline
        L_{4,1} & & & & t_2 & & & 1 & & -t_2 & & & & \\
        L_{4,2} & & & & & & & & 1 & & -t_2 & & & \\
        L_{3,2} & & & t_2 & & & & & & 1 & & -t_2 & & \\
        L_{2,2} & & & & & -t_2 & & & t_2 & & 1 & & & \\
        L_{1,2} & & & & -t_2 & & & & & t_2 & & 1 & & \\
        L_{1,3} & & & & & & & & & & t_2 & & 1 & \\
        L_{1,4} & & & -t_2 & & & & & & & & t_2 & & 1
        \end{bNiceArray}.
    \end{equation*}
    Here, $\det M=-t_3^6+(\text{lower degree terms in $t_3$})$, and the change of coordinates (\ref{eq:change of coordinates 1}) does not change the rank of $d(m_2)_{\tau_2}(X_2,X_3)$. Hence, $\operatorname{rank} (d(m_2)_{\tau_2}(X_2,X_3)) \geq 13$ for generic $(t_2,t_3) \in \mathbb{C}^{\times 2}$. This implies that $\dim \mu_2(\mathcal{T}_4)\geq 13$ by Lemma \ref{lemma:the rank of a differential gives a lower bound of dimension}. Combining with Lemma \ref{lemma:upper bound of the dimension, toeplitz}, we obtain that $\dim \mu_2(\mathcal{T}_4)=13$.
    
    \begin{figure}[t!]
    \includegraphics[width=0.3\textwidth]{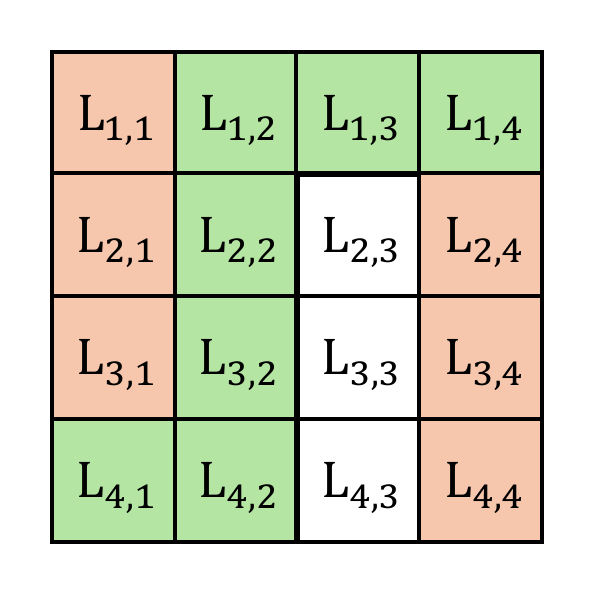}
    \caption{The set at (\ref{eq:picking Lij 1}) is the union of $\{L_{1,1},L_{2,1},L_{3,1}\}, \{L_{2,4},L_{3,4},L_{4,4}\}$, and $\{L_{4,1},L_{4,2},L_{3,2},L_{2,2},L_{1,2},L_{1,3},L_{1,4}\}$.} 
    \label{fig:picking Lij 1}
    \end{figure}
\end{example}

\begin{example}
    Consider the case where $(n,r)=(6,3)$. We calculate the dimension of $\mu_3(\mathcal{T}_6) \subsetneq \mathbb{C}^{6 \times 6}$. Consider the matrix multiplication map $m_3:\mathcal{T}_6 \times \mathcal{T}_6 \times \mathcal{T}_6 \rightarrow \mathbb{C}^{6 \times 6}$. From (\ref{eq:special points of pencils of Toeplitz matrices}) and (\ref{eq:the special point for differential}), we obtain
    \begin{equation*}
        T_3=B_0+t_3(B_3-B_{-3}),~T_4=B_0+t_4(B_4-B_{-4}),~T_5=B_0+t_5(B_5-B_{-5}),~\text{and}~\tau_3=(T_3,T_4,T_5).
    \end{equation*}
    Let $X_3,X_4$ and $X_5$ be matrices such that $[X_{k}]_{i,j}:=x_{k,j-i}$ for $k=3,4,5$ where the variables $x_{k,j-i}$'s are all independent. Then we obtain that
    \begin{equation*}
        d(m_3)_{\tau_3}(X_3,X_4,X_5)=X_3T_4T_5+T_3X_4T_5+T_3T_4X_5.
    \end{equation*}
    Consider the change of coordinates defined by 
    \begin{equation}\label{eq:change of coordinates 2}
        \begin{bmatrix}
            y_j\\
            z_{5,j}\\
            z_{4,j}
        \end{bmatrix}=\begin{bmatrix}
            1 & 1 & 1\\
            1 & 1 & 0\\
            1 & 0 & 0
        \end{bmatrix}\begin{bmatrix}
            x_{3,j}\\ x_{4,j} \\ x_{5,j}
        \end{bmatrix},~\text{ i.e., }~\begin{cases}
            y_j=x_{3,j}+x_{4,j}+x_{5,j}\\
            z_{5,j}=x_{3,j}+x_{4,j}\\
            z_{4,j}=x_{3,j}.
        \end{cases}
    \end{equation}
    Then 
    \begin{equation*}
        \begin{aligned}
            &d(m_3)_{\tau_3}(X_3,X_4,X_5)\\
            &=\begin{bmatrix}
                y_0 & y_1 & y_2 & y_3 & y_4 & y_5\\
                y_{-1} & y_0 & y_1 & y_2 & y_3 & y_4\\
                y_{-2} & y_{-1} & y_0 & y_1 & y_2 & y_3\\
                y_{-3} & y_{-2} & y_{-1} & y_0 & y_1 & y_2\\
                y_{-4} & y_{-3} & y_{-2} & y_{-1} & y_0 & y_1\\
                y_{-5} & y_{-4} & y_{-3} & y_{-2} & y_{-1} & y_0
            \end{bmatrix}\\
            &+t_5\begin{bmatrix}
                -z_{5,5} & 0 & 0 & 0 & 0 & z_{5,0}\\
                -z_{5,4} & 0 & 0 & 0 & 0 & z_{5,-1}\\
                -z_{5,3} & 0 & 0 & 0 & 0 & z_{5,-2}\\
                -z_{5,2} & 0 & 0 & 0 & 0 & z_{5,-3}\\
                -z_{5,1} & 0 & 0 & 0 & 0 & z_{5,-4}\\
                -z_{5,0} & 0 & 0 & 0 & 0 & z_{5,-5}
            \end{bmatrix}+t_4\begin{bmatrix}
                    -z_{4,4} & -z_{4,5} & 0 & 0 & z_{4,0} & z_{4,1}\\
                    -z_{4,3} & -z_{4,4} & 0 & 0 & z_{4,-1} & z_{4,0}\\
                    -z_{4,2} & -z_{4,3} & 0 & 0 & z_{4,-2} & z_{4,-1}\\
                    -z_{4,1} & -z_{4,2} & 0 & 0 & z_{4,-3} & z_{4,-2}\\
                    -z_{4,0} & -z_{4,1} & 0 & 0 & z_{4,-4} & z_{4,-3}\\
                    -z_{4,-1} & -z_{4,0} & 0 & 0 & z_{4,-5} & z_{4,-4}
                \end{bmatrix}\\
                &+(\text{the homogeneous linear terms in}~t_3)+(\text{higher degree terms in}~t_3,t_4,t_5)\\
                &=:L.
        \end{aligned}
    \end{equation*}
    \begin{figure}[t!]
    \includegraphics[width=0.45\textwidth]{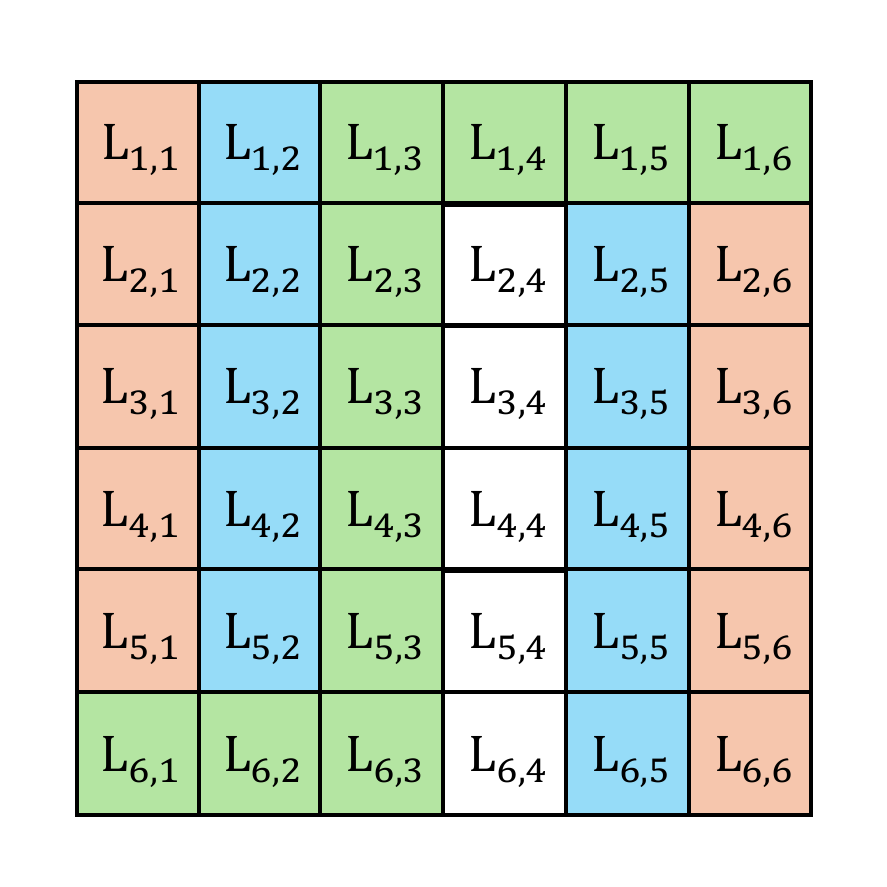}
    \caption{The set at (\ref{eq:picking Lij 2}) is the union of $\{L_{1,1},...,L_{5,1}\}, \{L_{2,6},...,L_{6,6}\}$, $\{L_{1,2},...,L_{5,2}\}$, $\{L_{2,5},...,L_{6,5}\}$ and $\{L_{6,1},L_{6,2},L_{6,3},L_{5,3},L_{4,3},L_{3,3},L_{2,3},L_{1,3},L_{1,4},L_{1,5},L_{1,6}\}$.} 
    \label{fig:picking Lij 2}
    \end{figure}
    If we select the positions
    \begin{equation}\label{eq:picking Lij 2}
    \begin{aligned}
        \{L_{1,1},L_{2,1},...,L_{5,1},L_{2,6},L_{3,6},...,L_{6,6},&L_{1,2},L_{2,2},...,L_{5,2},L_{2,5},L_{3,5},...,L_{6,5},\\
        &~L_{6,1},L_{6,2},L_{6,3},L_{5,3},L_{4,3},L_{3,3},L_{2,3},L_{1,3},L_{1,4},L_{1,5},L_{1,6}\}
    \end{aligned}
    \end{equation}
    as in Figure \ref{fig:picking Lij 2} and the variables
    \begin{equation*}
        \{z_{5,5},z_{5,4},...,z_{5,1},z_{5,-1},z_{5,-2},...,z_{5,-5},z_{4,5},z_{4,4},...,z_{4,1},z_{4,-1},z_{4,-2},...,z_{4,-5},y_{-5},y_{-4},...,y_0,...,y_4,y_5\}
    \end{equation*}
    then, with respect to these choices of column and row indices, we obtain the following submatrix $M$ of a matrix representation of $d(m_3)_{\tau_3}$:
    \begin{equation*}
        M=\begin{bmatrix}
            -t_5I_5  & * &* & * & *\\
            * & t_5I_5 & * & * & *\\
            * & * & -t_4I_5 & * & *\\
            * & * & * & t_4I_5 & *\\
            * & * & * & * & I_{11}
        \end{bmatrix}
    \end{equation*}
    if we present the diagonal entries of degree at most $1$ of each diagonal block. 
    
    We claim that, when calculating $\det M$ by expanding the Leibniz formula, a nonzero monomial with variable part
    \begin{equation*}
        t_5^{10}t_4^{10}
    \end{equation*}
    appears only once so that the monomial cannot be cancelled out. When expanding the Leibniz formula, if we only keep track of contribution from the terms of degree at most $1$ in all factors, then such nonzero monomial  appears only once for the following reason. There is only one way to obtain the factor $t_5^{10}$, which is taking the diagonal elements of the upper-left corner of the submatrix $M$, because $t_5$ only appears in the positions $L_{p,1}$ and $L_{p,6}$. After obtaining the factor $t_5^{10}$ in this way, there is only way to obtain the factor $t_4^{10}$, which is taking the diagonal elements of the upper-left corner of the remaining submatrix in a similar reason. Finally, there is only way to obtain a nonzero constant from the remaining submatrix, which is taking the diagonal entries of the remaining $11 \times 11$ submatrix. Thus, there is only one way to make such monomial if we use only the terms of degree at most $1$. If we want to obtain such nonzero monomial using an entry of degree at least $2$, then we need to choose at least $12$ terms of $1$, but it is impossible because all entries containing $1$ in $M$ are on the last $11$ columns and last $11$ rows. Therefore, the claim is true. 
    
    The claim implies that $\det M$ is a nonzero polynomial. Note that the change of coordinates (\ref{eq:change of coordinates 2}) does not change the rank of $d(m_3)_{\tau_3}(X_3,X_4,X_5)$. Hence, $\operatorname{rank} ((dm_3)_{\tau_3}(X_3,X_4,X_5)) \geq 31$ for generic $(t_3,t_4,t_5) \in \mathbb{C}^{\times 3}$. This implies that $\dim \mu_3(\mathcal{T}_6) \geq 31$ by Lemma \ref{lemma:the rank of a differential gives a lower bound of dimension}. Combining with Lemma \ref{lemma:upper bound of the dimension, toeplitz}, we obtain that $\dim \mu_3(\mathcal{T}_6)=31$.
\end{example}

Now, we prove the main theorem of this section.

\begin{theorem}\label{thm:dimension for Toeplitz and Hankel structures}
    The dimension of $\mu_r(\mathcal{T}_n)$ (and so the dimension of $\mu_r(\mathcal{H}_n)$) satisfies the equality
    \begin{equation*}
        \dim \mu_r(\mathcal{T}_n) = \dim \mu_r(\mathcal{H}_n) = \min\{n^2, 2r(n-1)+1\}.
    \end{equation*}
\end{theorem}
\begin{proof}
    The case of $r=1$ is trivial, and the case of $r \geq \lfloor n/2 \rfloor +1$ is from Theorem \ref{thm:previous results}.(\romannumeral2). We only prove the case where $1<r<\lfloor n/2 \rfloor+1$. By the optimality of the number $\lfloor n/2 \rfloor+1$ (see Example \ref{ex:Toeplitz}), we obtain that $\dim \mu_r(\mathcal{T}_n) < n^2$. We will prove that $\dim \mu_r(\mathcal{T}_n)=2r(n-1)+1$. 
    
    Consider the special point $\tau_r=(T_{n-r},...,T_{n-1})$ described by (\ref{eq:special points of pencils of Toeplitz matrices}) and (\ref{eq:the special point for differential}). Let $X_{n-r},...,X_{n-1}$ be matrices such that $[X_k]_{i,j}:=x_{k,j-i}$ for $k=n-r,...,n-1$ where the variables $x_{k,j-i}$'s are all independent. Then the matrix multiplication map $m_r:X^{\times r} \rightarrow \mathbb{C}^{n \times n}$ has the differential 
\begin{equation*}
    d(m_r)_{\tau_r}(X_{n-r},...,X_{n-1})=\sum_{k=1}^{r}T_{n-r}\cdots T_{n-k-1}X_{n-k}T_{n-k+1}\cdots T_{n-1}
\end{equation*}
at the special point $\tau_r$. Let $L$ be the result matrix of $d(m_r)_{\tau_r}(X_{n-r},...,X_{n-1})$. Considering the change of coordinates defined by
\begin{equation}\label{eq:change of basis in main theorem}
    \begin{bmatrix}
        y_j\\
        z_{n-1,j}\\
        \vdots\\
        z_{n-r+2,j}\\
        z_{n-r+1,j}
    \end{bmatrix}=\begin{bmatrix}
            1 & 1 & \cdots & 1 & 1\\
            1 & 1 & \cdots & 1 & 0\\
            \vdots & \vdots & \iddots &\iddots & \vdots\\
            1 & 1 & 0 & \cdots & 0\\
            1 & 0 & 0 & \cdots & 0
        \end{bmatrix}\begin{bmatrix}
            x_{n-r,j}\\
            x_{n-r+1,j}\\
            \vdots\\
            x_{n-2,j}\\
            x_{n-1,j}
        \end{bmatrix}~\text{ i.e., }~\begin{cases}
            y_j=x_{n-r,j}+\cdots+x_{n-2,j}+x_{n-1,j}\\
            z_{n-1,j}=x_{n-r,j}+\cdots+x_{n-2,j}\\
            \qquad\qquad\vdots\\
            z_{n-r+2,j}=x_{n-r,j}+x_{n-r+1,j}\\
            z_{n-r+1,j}=x_{n-r,j},
        \end{cases}
\end{equation}
 Consider the basis $\{B_{-n+1},...,B_{-1},B_0,B_1,\cdots,B_{n-1}\}$ described at (\ref{eq:basis of the space of Toeplitz matrices}). From the definition of $T_{n-i}$ described at (\ref{eq:special points of pencils of Toeplitz matrices}), we obtain that 
    \begin{equation}\label{eq:calculation of a term of Lpq}
        \begin{aligned}
            &T_{n-r}\cdots T_{n-k-1}X_{n-k}T_{n-k+1}\cdots T_{n-1}\\
            &=X_{n-k}\\
            &+(t_{n-r}(B_{n-r}-B_{-(n-r)})+\cdots+t_{n-k-1}(B_{n-k-1}-B_{-(n-k-1)}))X_{n-k}\\
            &+X_{n-k}(t_{n-k+1}(B_{n-k+1}-B_{-(n-k+1)})+\cdots+t_{n-1}(B_{n-1}-B_{-(n-1)}))\\
            &+(\text{higher degree terms in $t_{n-r},...,t_{n-k-1},t_{n-k+1},...,t_{n-1}$})
        \end{aligned}
    \end{equation}
    for each $k=1,...,r$. 
    
    Here, $B_0$ is the identity matrix. When $k$ is positive, then  multiplication on the left (resp. right) by $B_k$ shifts the rows (resp. columns) of a given matrix upward (resp. to the right) by $k$ positions. When $k$ is negative, then multiplication on the left (resp. right) by $B_k$ shifts the rows (resp. columns) of a given matrix downward (resp. to the left) by $|k|$ positions. Thus, the equality (\ref{eq:calculation of a term of Lpq}) and the change of coordinate (\ref{eq:change of basis in main theorem}) imply that the matrix $L$ can be described as 
    \begin{equation*}
        L=Y+T_{\text{linear}}+T_{\text{higher}}
    \end{equation*}
    where the matrices $Y, T_{\text{linear}}$, and $T_{\text{higher}}$ satisfy the following properties:
    \begin{itemize}
        \item [(\romannumeral1)] $Y_{p,q}=y_{q-p}$;
        \item [(\romannumeral2)] For each $k=1,...,r-1$, the nonzero coefficient of $t_{n-k}$ occurs only on the first $k$ columns and last $k$ columns of $T_{\text{linear}}$, and the coefficients of $t_{n-k}$ of the $(n-p+1,k)$-entry and $(p,n-k+1)$-entry of $T_{\text{linear}}$ are $-z_{n-k,p-1}$ and $z_{n-k,-(p-1)}$, respectively;
        \item [(\romannumeral3)] $T_{\text{higher}}$ consists of higher degree terms in $t_{n-r},...,t_{n-1}$ (if they exist).
    \end{itemize}
    \begin{figure}[t!]
    \includegraphics[width=1\textwidth]{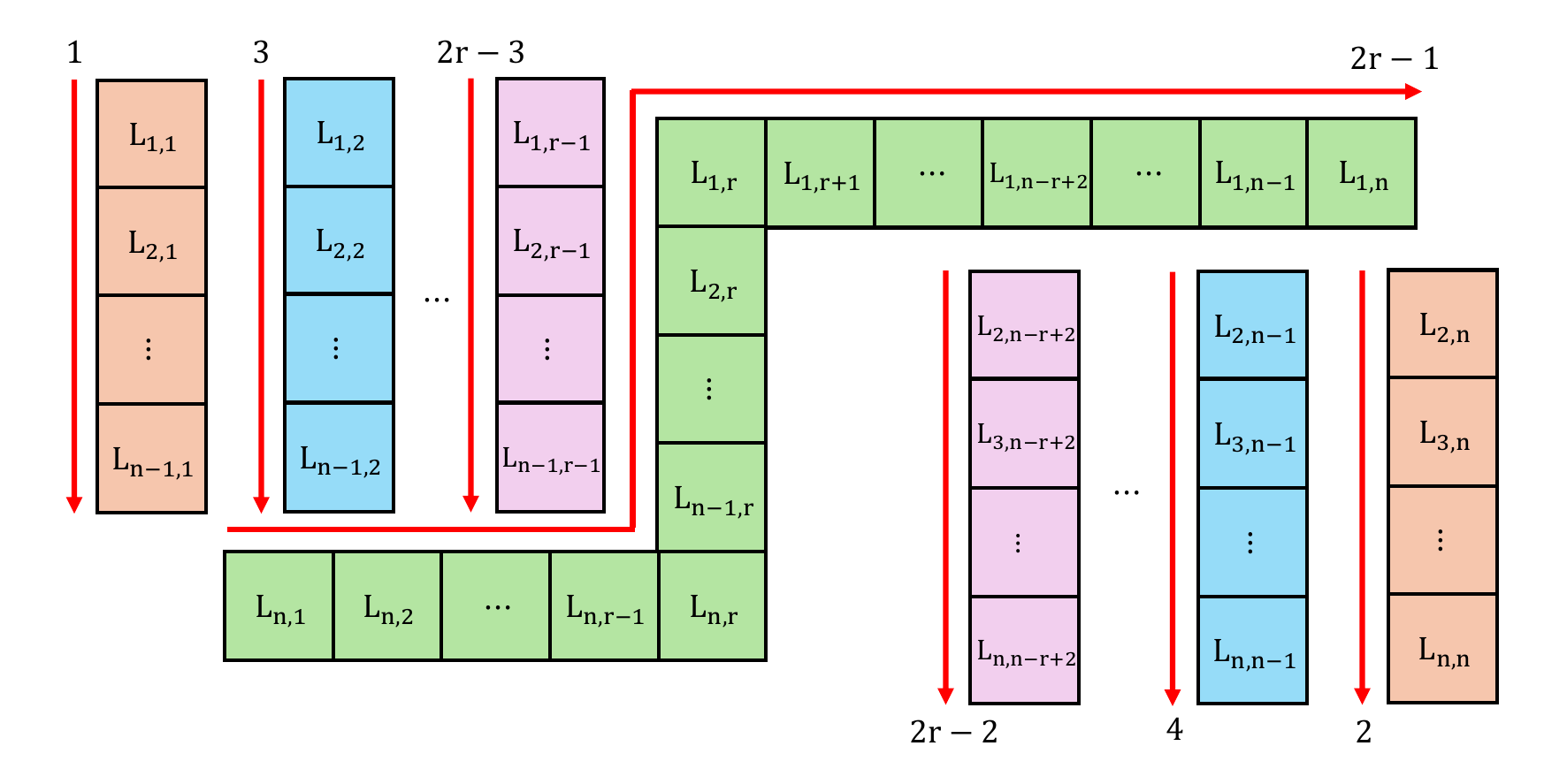}
    \caption{The way of selecting the ordered positions $L_{p,q}$: Select $\{L_{1,1}, L_{2,1},...,L_{n-1,1}\}$ on the first column, $\{L_{2,n},L_{3,n},...,L_{n,n}\}$ on the last column, $\{L_{1,2}, L_{2,2},...,L_{n-1,2}\}$ on the second column, $\{L_{2,n-1},L_{3,n-1},...,L_{n,n-1}\}$ on the second-to-last column. Keep similarly up to keeping $\{L_{1,r-1}, L_{2,r-1},...,L_{n-1,r-1}\}$ on the $(r-1)$-th column, $\{L_{2,n-r+2},L_{3,n-r+2},...,L_{n,n-r+2}\}$ on the $(n-r+2)$-th column, and $\{L_{n,1},L_{n,2},\cdots, L_{n,r-1}, L_{n,r}, L_{n-1,r},...,L_{2,r},L_{1,r},L_{1,r+1},...,L_{1,n}\}$ avoiding any intersection with previous selections and choosing only one position on each diagonal.} 
    \label{fig:picking Lij 3}
    \end{figure}
    We select the positions as in the following ordered set
    \begin{equation*}
    \begin{aligned}
        &(\{L_{1,1}, L_{2,1},...,L_{n-1,1}\} \sqcup \{L_{2,n},L_{3,n},...,L_{n,n}\})\\
        &\sqcup (\{L_{1,2}, L_{2,2},...,L_{n-1,2}\}\sqcup \{L_{2,n-1},L_{3,n-1},...,L_{n,n-1}\})\sqcup \cdots\\
        & \sqcup (\{L_{1,r-1}, L_{2,r-1},...,L_{n-1,r-1}\} \sqcup \{L_{2,n-r+2},L_{3,n-r+2},...,L_{n,n-r+2}\})\\
        &\sqcup \{L_{n,1},L_{n,2},\cdots, L_{n,r-1}, L_{n,r}, L_{n-1,r},...,L_{2,r},L_{1,r},L_{1,r+1},...,L_{1,n}\}
    \end{aligned}
    \end{equation*}
    as in Figure \ref{fig:picking Lij 3} for row indices. For $k=1,...,r-1$, let
    \begin{equation*}
        \mathcal{B}_{n-k}=\{z_{n-k,j}~|~j=n-1,n-2,...,1,-1,...,-(n-2),-(n-1)\}
    \end{equation*}
    and let
    \begin{equation*}
        \mathcal{B}_0=\{y_{j}~|~-(n-1) \leq j \leq (n-1)\},
    \end{equation*}
    be ordered sets. Take
    \begin{equation*}
        \mathcal{B}_{n-1} \sqcup \mathcal{B}_{n-2} \sqcup \cdots \sqcup \mathcal{B}_{n-(r-1)} \cup \mathcal{B}_0
    \end{equation*}
    as an ordered set for column indices. Then, with respect to these choices of column and row indices, we obtain the following $(2r(n-1)+1) \times (2r(n-1)+1)$ submatrix $M$ of a matrix representation of the linear map $d(m_r)_{\tau_r}$:
    \begin{equation*}
        M=\left[\begin{array}{cc|cc|c|cc|c}
            -t_{n-1}I_{n-1} & * & * & * & \cdots   & * &  * & * \\
            * & t_{n-1}I_{n-1} & * & * & \cdots  & * &  * & * \\ \hline
            * & * & -t_{n-2}I_{n-1} & * & \cdots & * & * & * \\
            * & * & * & t_{n-2}I_{n-1} & \cdots & * & * & * \\ \hline
            \vdots & \vdots &  \vdots & \vdots & \ddots & \vdots & \vdots & \vdots \\ \hline
            * & * & * & * & \cdots & -t_{n-(r-1)}I_{n-1} & * & *\\
             * & * & * & * & \cdots & * & t_{n-(r-1)}I_{n-1} & *\\ \hline
             * & * & * & * &\cdots & * & * & I_{2n-1}
        \end{array}\right]
    \end{equation*}
     if we present the diagonal entries of degree at most $1$ of each diagonal block. 
    
    We claim that, when calculating $\det M$ by expanding the Leibniz formula, a nonzero monomial with variable part
    \begin{equation*}
        t_{n-1}^{2n-2}t_{n-2}^{2n-2}\cdots t_{n-(r-1)}^{2n-2}
    \end{equation*}
    appears only once so that the monomial cannot be cancelled out. When expanding the Leibniz formula, if we only keep track of contribution from the terms of degree at most $1$ in all factors, then such nonzero monomial appears only once for the following reason. There is only one way to obtain the factor $t_{n-1}^{2n-2}$, which is taking the diagonal entries of the upper-left corner of the submatrix $M$, because $t_{n-1}$ only appears in the positions $L_{p,1}$ and $L_{p,n}$. After obtaining the factor $t_{n-1}^{2n-2}$ in this way, there is only way to obtain the factor $t_{n-2}^{2n-2}$, which is taking the diagonal entries of the upper-left corner of the remaining submatrix in a similar reason. After obtaining the factor $t_{n-(r-1)}^{2n-2}$ from these successive arguments, there is only way to obtain a nonzero constant from the remaining submatrix, which is taking the diagonal entries on the remaining $(2n-1) \times (2n-1)$ submatrix. Thus, there is only one way to make such monomial if we use only the terms of degree at most $1$. If we want to obtain such nonzero monomial using an entry of degree at least $2$, then we need to choose at least $2n$ terms of $1$, but it is impossible because all entries containing $1$ in $M$ are on the last $2n-1$ columns and last $2n-1$ rows. Therefore, the claim is true, and the monomial must be $(-1)^{(r-1)(n-1)}t_{n-1}^{2n-2}t_{n-2}^{2n-2}\cdots t_{n-(r-1)}^{2n-2}$. 

    The claim implies that $\det M$ is a nonzero polynomial. Note that the change of coordinates (\ref{eq:change of basis in main theorem}) does not change the rank of $d(m_r)_{\tau_r}(X_{n-r},...,X_{n-1})$. Hence, we have that 
    \begin{equation*}
        \rank(d(m_r)_{\tau_r}(X_{n-r},...,X_{n-1})) \geq 2r(n-1)+1
    \end{equation*}
    for generic $(t_{n-r},...,t_{n-1}) \in \mathbb{C}^{\times r}$. This implies that $\dim \mu_r(\mathcal{T}_n) \geq 2r(n-1)+1$ by Lemma \ref{lemma:the rank of a differential gives a lower bound of dimension}. Combining with Lemma \ref{lemma:upper bound of the dimension, toeplitz}, we obtain that $\dim \mu_r(\mathcal{T}_n)=2r(n-1)+1$.
\end{proof}

\subsubsection{Further Examples}

\begin{example}[Bidiagonal matrix structure]\label{ex:dimension for bidiagonal structure}
    As we mentioned in Example \ref{ex:bidiagonal}, for $2 \leq k \leq n$, $\mu_{k-1}(\mathcal{UBD}_n)$ (resp. $\mu_{k-1}(\mathcal{LBD}_n)$) is the affine variety of $n \times n$ upper (resp. lower) $k$-diagonal matrices. Hence, for $2 \leq k \leq n$, we obtain that
    \begin{equation*}
        \dim \mu_{k-1}(\mathcal{UBD}_n)= \dim \mu_{k-1}(\mathcal{LBD}_n)= n+(n-1)+\cdots+(n-k+1)=kn-\frac{k(k-1)}{2}.
    \end{equation*}
\end{example}

In \cite{ye2017new}, the author proved that $\mu_{k-1}(\mathcal{UBD}_n)$ is the affine variety of $n \times n$ upper $k$-diagonal matrices by induction and Proposition \ref{prop:the rank of differential and dominance}. We will prove the following theorem for the tridiagonal matrix structure in a similar way. A matrix $A=[a_{ij}]$ is said to be $(2k+1)$-diagonal matrix if $a_{ij}=0$ if $|i-j| \geq k+1$. 

\begin{theorem}[Tridiagonal matrix structure]\label{thm:dimension for tridiagonal structure}
    Consider $\mathcal{TD}_n$ the variety of $n \times n$ tridiagonal matrices. Let $1 \leq r \leq n-1$. Then $\mu_{r}(\mathcal{TD}_n)$ is the variety of $n \times n$ $(2r+1)$-diagonal matrices. In particular,
    \begin{equation*}
        \dim \mu_{r}(\mathcal{TD}_n)=n+2\sum_{l=1}^{r}(n-l),
    \end{equation*}
    and hence $\mathbf{FL}_{\text{gen}}(\mathcal{TD}_n)=n-1$ and $\mathbf{FL}_{\text{max}}(\mathcal{TD}_n)\leq 4(n-1)$.
\end{theorem}
\begin{proof}
    Let $\mathcal{D}_{2k+1}$ denote the affine variety of $n \times n$ $(2k+1)$-diagonal matrices. Note that, for $1 \leq r \leq n-1$, a product of $r$ $n \times n$ tridiagonal matrices is in $\mathcal{D}_{2r+1}$. We prove that the matrix multiplication map $m_r:\mathcal{TD}_n^{\times r} \rightarrow \mathcal{D}_{2r+1}$ is dominant by induction on $r$. If $r=1$, the identity map $m_1:\mathcal{TD}_n \rightarrow \mathcal{TD}_n$ is clearly dominant. Assume that the map $m_{r-1}$ is dominant, where $1 \leq r-1 \leq n-2$. The map $m_r$ can be considered as the composition
    \begin{equation*}
        m_r:\mathcal{TD}_n \times \mathcal{TD}_n^{\times (r-1)} \xrightarrow{(id,m_{r-1})}\mathcal{TD}_n \times \mathcal{D}_{2r-1} \xrightarrow{m} \mathcal{D}_{2r+1}, 
    \end{equation*}
    where $m$ denotes the matrix multiplication map. By the induction hypothesis, $(id, m_{r-1})$ is dominant. We show that $m$ is dominant, by considering its differential 
    \begin{equation*}
        dm_{(A,B)}:(X,Y) \mapsto XB+AY.
    \end{equation*}
    at  $(A,B) \in \mathcal{TD}_{n} \times \mathcal{D}_{2r-1}$ for 
    \begin{equation*}
        A=I_n~\text{ and }~B=[b_{ij}]~\text{ where }~b_{ij}=\begin{cases}
            1 & \text{if }|i-j|=r-1\\
            0 & \text{otherwise}.
        \end{cases}
    \end{equation*}
    Let $M \in \mathcal{D}_{2r+1}$. Then $M=M'+M''$, where $M' \in \mathcal{D}_{2r-1}$ and $M'' \in \mathcal{D}_{2r+1} \setminus \mathcal{D}_{2r-1}$. Since $M''\in \mathcal{D}_{2r+1} \setminus \mathcal{D}_{2r-1}$, then $M''$ has only nonzero entries on upper and lower $r$-th diagonals. That is, for $M''=[m_{ij}'']$, only the $2(n-r)$ entries $m_{1,r+1}'',m_{2,r+2}'',...,m_{n-r,n}''$ and $m_{r+1,1}'',m_{r+2,2}'',...,m_{n,n-r}''$ can be nonzero. If we let $X=[x_{ij}]$ satisfy
    \begin{equation*}
        x_{1,2}=m_{1,r+1}'',~ x_{2,3}=m_{2,r+2}'',~..., ~x_{n-r,n-r+1}=m_{n-r,n}'',
    \end{equation*}
    and 
    \begin{equation*}
        x_{r+1,r}=m_{r+1,1}'',~ x_{r+2,r+1}=m_{r+2,2}'',~...,~ x_{n,n-1}=m_{n,n-r}'',
    \end{equation*}
    then we have that $XB=M''+E$ for some $E \in D_{2r-1}$. Hence, if we let $Y=M'-E$, then $dm_{(A,B)}(X,Y)=M$. Thus, the $dm_{(A,B)}$ is surjective, and so $m$ is dominant by Proposition \ref{prop:the rank of differential and dominance}. 
    
    By Proposition \ref{prop:the number of structured matrices needed to express every matrix}, $\mathbf{FL}_{\text{gen}}(\mathcal{TD}_n)=n-1$ implies every $n \times n$ matrix is a product of $4(n-1)$ tridiagonal matrices and a diagonal matrix. Since a product of a diagonal matrix and a tridiagonal matrix is tridiagonal, then $\mathbf{FL}_{\text{max}}(\mathcal{TD}_n)\leq 4(n-1)$.
\end{proof}

\begin{theorem}[Skew-symmetric matrix structure] 
    Consider $\Lambda_n$ the affine variety of $n \times n$ skew-symmetric matrices. For each $n,r \geq 1$, the dimension of $\mu_r(\Lambda_n)$ is as in Table \ref{table:dimensions, skew-symmetric structure}.
    \begin{table}[t!]
        \centering
        \begin{tabular}{|c|c|c|}
            \hline
            $n$ & $r$ & $\dim \mu_r(\Lambda_n)$ \\ \hline
            $2$ & $\geq 1$ & $1$ \\ \hline
            $\geq 3$ & $1$ & $n(n-1)/2$ \\ \hline
            $\geq 3$ & $2$ & $n^2-\lfloor 3n/2 \rfloor$ \\ \hline
            $3$ & $3$ & $7$ \\
            $3$ & $\geq 4$ & $8$\\ \hline
            $4$ & $3$ & $13$\\
            $4$ & $4$ & $15$\\
            $4$ & $\geq 5$ & $16$\\ \hline
            $6$ & $3$ & $35$ \\
            $6$ & $\geq 4$ & $36$ \\ \hline 
            $(\text{odd})\geq 5$ & $\geq 3$ & $n^2-1$\\ \hline
            $(\text{even})\geq 8$ & $\geq 3$ & $n^2$ \\ \hline
        \end{tabular}
        \caption{The dimensions of $\mu_r(\Lambda_n)$ for all $n,r$.}
        \label{table:dimensions, skew-symmetric structure}
    \end{table}
\end{theorem}
\begin{proof}
    Considering Theorem \ref{thm:previous results} and the results in \cite[Example 5.6]{ye2017new}, it suffices to calculate $\dim \mu_2(\Lambda_n)$ for each $n \geq 1$. We claim that the image of the matrix multiplication map $m_2:\Lambda_n^{\times 2} \rightarrow \mathbb{C}^{n \times n}$ has dimension $n^2-\lfloor 3n/2 \rfloor$. Consider its differential 
    \begin{equation*}
        d(m_2)_{(A,B)}:(X,Y) \mapsto XB+AY
    \end{equation*}
    at a point $(A,B) \in \Lambda_n^{\times 2}$. We prove the claim when $n$ is even, because the case of odd $n$ can be proved similarly. 
    
    Let $n=2m$ for some $m \geq 1$. Let $J_{2m}$ be the block diagonal matrix with all diagonal blocks being $J_2$, where
        \begin{equation*}
            J_2=\begin{bmatrix}
                0 & 1\\
                -1 & 0
            \end{bmatrix},
        \end{equation*}
        that is,
        \begin{equation*}
            J_{2m}=\left[\begin{array}{cc|c|cc}
                0 & 1 & & & \\
                -1 & 0 & & & \\ \hline
                & & \ddots & & \\ \hline
                & & & 0 & 1\\
                & & & -1  & 0
            \end{array}\right].
        \end{equation*}
        Let $D$ be the block diagonal matrix with $i$-th  diagonal block being $\lambda_i I_2$, that is,
        \begin{equation}\label{eq:diagonal matrix in the argument of skew-symmetric matrix structure}
            D=\left[\begin{array}{cc|c|cc}
                \lambda_1 & 0 & & & \\
                0 & \lambda_1 & & & \\ \hline
                & & \ddots & & \\ \hline
                & & & \lambda_m & 0\\
                & & & 0  & \lambda_m
            \end{array}\right].
        \end{equation}
        Let
        \begin{equation*}
            A=J_{2m}~\text{ and } B=-J_{2m}D.
        \end{equation*}
        We calculate the kernel of $d(m_2)_{(A,B)}$:
        \begin{equation*}
            \{(X,Y) \in \Lambda_{2m}^{\times 2}~|~d(m_2)_{(A,B)}(X,Y)=0\}.
        \end{equation*}
        Consider the skew-symmetric matrices $X$ and $Y$ as $m \times m$ block matrices with $(i,j)$-block being the $2 \times 2$ matrix $X_{ij}$ and $Y_{ij}$, respectively. Then we obtain the equivalences:
        \begin{equation}\label{eq:equivalences}
        \begin{aligned}
            &d(m_2)_{(A,B)}(X,Y)=0\\
            &\Leftrightarrow -XJD+JY=0\\
            &\Leftrightarrow -\lambda_j X_{ij}J_2 + J_2Y_{ij}=0,~\text{ for all }~1 \leq i,j \leq m.
        \end{aligned}
        \end{equation}
        Since $X$ and $Y$ are skew-symmetric matrices, then we have that $X_{ji}=-X_{ij}^T$ and $Y_{ji}=-Y_{ij}^T$. Hence, from the last equalities on (\ref{eq:equivalences}), we obtain that
        \begin{equation*}
            (\lambda_i-\lambda_j)X_{ij}=0~\text{ for all }~1 \leq i,j \leq m.
        \end{equation*}
        This implies that $X_{ij}=0$ whenever $i \neq j$, and hence $X$ must be a block diagonal matrix. Note that $Y$ is completely determined by $X$. Thus, the kernel of $d(m_2)_{(A,B)}$ has dimension $m$. By Lemma \ref{lemma:the rank of a differential gives a lower bound of dimension}, we obtain that
        \begin{equation*}
            \dim \mu_2(\Lambda_{2m}) \geq \rank(d(m_2)_{(A,B)}) \geq 2 \cdot \frac{2m(2m-1)}{2}-m=4m^2-3m.
        \end{equation*}

        Now, we prove that $\dim \mu_2(\Lambda_{2m}) \leq 4m^2-3m$. Note that the set $\{A'B'~|~A'~\text{is invertible in}~\Lambda_{2m},~B' \in \Lambda_{2m}\}$ is a Zariski dense subset of $\mu_2(\Lambda_{2m})$, and let $U$ be a nonzero subset of this set which is Zariski open in $\mu_2(\Lambda_{2m})$. An element $A'B' \in U$ ($A',B' \in \Lambda_{2m}$) has the characteristic polynomial
        \begin{equation*}
            \chi_{A'B'}(t)=\det(tI_{2m}-A'B')=\det(A')\det(t(A')^{-1}-B').
        \end{equation*}
        Here, $t(A')^{-1}-B'$ is a skew-symmetric matrix, and so $\det(t(A')^{-1}-B')$ is the square of its Pfaffian. This implies that the matrix $A'B'$, that is, a generic element of $\mu_2(\Lambda_{2m})$ is conjugate to a diagonal matrix $D$ described at (\ref{eq:diagonal matrix in the argument of skew-symmetric matrix structure}), and so is in the set
        \begin{equation*}
            \mathcal{O}:=\{gDg^{-1}~|~g \in \operatorname{GL}_{2m},~D~\text{is of the form (\ref{eq:diagonal matrix in the argument of skew-symmetric matrix structure})}\}.
        \end{equation*}
        For a fixed $D$ of the form (\ref{eq:diagonal matrix in the argument of skew-symmetric matrix structure}), the stabilizer subgroup of $\operatorname{GL}_{2m}$ is $\operatorname{GL}_2^{\times m}$, and so the orbit $\{gDg^{-1}~|~g \in \operatorname{GL}_{2m}\}$ has dimension $4m^2-4m$ by the orbit-stabilizer dimension theorem (see \cite[Corollary 5.5.6]{MR1642713}). Since there are $m$ parameters for the choice of $D$, we have that
        \begin{equation*}
            \dim \mathcal{O} =4m^2-4m+m =4m^2-3m.
        \end{equation*}
        Since a generic element of $\mu_2(\Lambda_{2m})$ is in $\mathcal{O}$, then we have that
        \begin{equation*}
            \dim \mu_2(\Lambda_{2m}) \leq \dim \mathcal{O} = 4m^2-3m.
        \end{equation*}
        Therefore, $\dim \mu_2(\Lambda_{2m}) = 4m^2-3m$.

        If $n=2m+1$ for some $m \geq 1$, then we can obtain that $\dim \mu_2(\Lambda_{2m+1})=(2m+1)^2-(4m+1)+m$ in a similar way. 
\end{proof}

\begin{theorem}[Companion matrix structure]\label{thm:dimension for companion structure}
    Consider $\mathcal{C}_n$ the affine variety of $n \times n$ companion matrices. For each $n,r \geq 1$, the dimension of $\mu_r(\mathcal{C}_n)$ is
    \begin{equation*}
        \dim \mu_r(\mathcal{C}_n)=\min\{rn,n^2\}.
    \end{equation*}
\end{theorem}
\begin{proof}
    Although $\mathcal{C}_n$ does not contain the identity matrix $I_n$, $\mu_r(\mathcal{C}_n)=\mathbb{C}^{n \times n}$ implies $\mu_{r+1}(\mathcal{C}_n)=\mathbb{C}^{n \times n}$, because there is an invertible matrix $C \in \mathcal{C}_n$: for any $A \in \mathbb{C}^{n \times n}$, it can be expressed as $A=(AC^{-1})C$, where $AC^{-1} \in \mathbb{C}^{n \times n}$ and $C \in \mathcal{C}_n$. Hence, we only prove this theorem for $1 \leq r \leq n$, where $n=\mathbf{FL}_{\text{gen}}(\mathcal{C}_n)$ (by Theorem \ref{thm:previous results}.(\romannumeral7) and Example \ref{ex:companion matrix structure}).
    
    Let $\mathcal{V}_k$ denote the affine variety of $n \times n$ matrices of the following form
    \begin{equation*}
        \begin{bmatrix}
            0 & 0 &\cdots & 0 & c_{1,n-k+1} & \cdots & c_{1,n}\\
            \vdots & \vdots & \ddots & \vdots & \vdots & \ddots & \vdots\\
            0 & 0 &\cdots & 0 & c_{n-k,n-k+1} & \cdots & c_{n-k,n}\\
            1 & 0 & \cdots & 0 & c_{n-k+1,n-k+1} & \cdots & c_{n-k+1,n}\\
            0 & 1 & \cdots & 0 & c_{n-k+2,n-k+1} & \cdots & c_{n-k+2,n}\\
            \vdots & \vdots & \ddots & \vdots & \vdots & \ddots & \vdots\\
             0 & 0 & \cdots & 1 & c_{n,n-k+1} & \cdots & c_{n,n}\\
        \end{bmatrix}.
    \end{equation*}
    Note that, for $1 \leq r \leq n$, a product of $r$ $n \times n$ companion matrices is in $\mathcal{V}_r$. We prove that the matrix multiplication map $m_r:\mathcal{C}_n^{\times r} \rightarrow \mathcal{V}_{r}$ is dominant by induction on $r$. If $r=1$, the identity map $m_1:\mathcal{C}_n \rightarrow \mathcal{C}_n$ is clearly dominant. For $1 \leq r-1 \leq n-1$, assume that the map $m_{r-1}$ is dominant, and let $U$ be a nonempty subset of the image  $m_{r-1}(\mathcal{C}_n^{\times (r-1)})$ which is Zariski open in $\mathcal{V}_{r-1}$. Let 
    \begin{equation*}
        A_r := \left[\begin{array}{ccc|ccc|c}
             0  &\cdots & 0 & a_{1,n-r+1} & \cdots & a_{1,n-1} & a_{1,n}\\
            \vdots  & \ddots & \vdots & \vdots & \ddots & \vdots & \vdots\\
            0  &\cdots & 0 & a_{n-r,n-r+1} & \cdots & a_{n-r,n-1} & a_{n-r,n}\\
            1  & \cdots & 0 & a_{n-r+1,n-r+1} & \cdots & a_{n-r+1,n-1} & a_{n-r+1,n}\\
            \vdots  & \ddots & \vdots & \vdots & \ddots & \vdots & \vdots \\
             0 &  \cdots & 1 & a_{n,n-r+1} & \cdots & a_{n,n-1} & a_{n,n}\\
        \end{array}\right] \in \mathcal{V}_r
    \end{equation*}
    satisfy that
    \begin{equation*}
        A_{r-1}:=\left[\begin{array}{cccc|ccc}
             0 & 0 &\cdots & 0 & a_{1,n-r+1} & \cdots & a_{1,n-1} \\
            \vdots & \vdots  & \ddots & \vdots & \vdots & \ddots & \vdots \\
            0 & 0  &\cdots & 0 & a_{n-r-1,n-r+1} & \cdots & a_{n-r-1,n-1} \\
            1 & 0  &\cdots & 0 & a_{n-r,n-r+1} & \cdots & a_{n-r,n-1} \\
            0  & 1 & \cdots & 0 & a_{n-r+1,n-r+1} & \cdots & a_{n-r+1,n-1}\\
            \vdots & \vdots & \ddots & \vdots & \vdots & \ddots & \vdots \\
             0 & 0 &  \cdots & 1 & a_{n,n-r+1} & \cdots & a_{n,n-1} \\
        \end{array}\right] \in \mathcal{V}_{r-1}
    \end{equation*}
    is an invertible element of $U$, so that we can say $A_r$ is a generic element in $\mathcal{V}_r$. Since $A_{r-1} \in U$, then $A_{r-1}=C_1\cdots C_{r-1}$ for some $C_1,...,C_{r-1} \in \mathcal{C}_n$. Furthermore, since $A_{r-1}$ is invertible, we can take a companion matrix 
    \begin{equation*}
        C_r:= \begin{bmatrix}
        0 & 0 & \cdots & 0 & c_1\\
        1 & 0 & \cdots & 0 & c_2\\
        \vdots & \vdots & \ddots & \vdots & \vdots \\
        0 & 0 & \cdots & 1 & c_n
    \end{bmatrix}.
    \end{equation*}
    such that
    \begin{equation*}
        A_{r-1}\begin{bmatrix}
            c_1 \\ c_2 \\ \vdots \\ c_n
        \end{bmatrix}=\begin{bmatrix}
            a_{1,n} \\ a_{2,n} \\ \vdots \\ a_{n,n}
        \end{bmatrix}
    \end{equation*} 
    and hence
    \begin{equation*}
        A_r=A_{r-1}C_r=C_1\cdots C_{r-1}C_r.
    \end{equation*}
    Thus, a generic element in $\mathcal{V}_r$ is a product of $r$ companion matrices, that is, $m_r:\mathcal{C}_n^{\times r} \rightarrow \mathcal{V}_r$ is dominant.
\end{proof}

\begin{remark}
    By \cite[Lemma 5.1]{ye2017new} (see Example \ref{ex:dimension for bidiagonal structure}) and the proofs of Theorems \ref{thm:dimension for tridiagonal structure} and \ref{thm:dimension for companion structure}, we can determine not only the dimension but also the defining ideal of $\mu_r(X)$ for all $r \geq 1$, for $X$ being $\mathcal{UBD}_n$, $\mathcal{LBD}_n$, $\mathcal{TD}_n$, or $\mathcal{C}_n$. Note that the defining ideals are all generated by affine linear polynomials in these cases.
\end{remark}

\subsection{Degree computations via numerical algebraic geometry}\label{subsection:Degree}

When an affine variety is defined by some affine linear polynomials, then this affine variety must have the degree $1$. Thus, from \cite[Lemma 5.1]{ye2017new} (see Example \ref{ex:dimension for bidiagonal structure}) and the proofs of Theorems \ref{thm:dimension for tridiagonal structure} and \ref{thm:dimension for companion structure}, we obtain that 
\begin{equation*}
    \deg \mu_r(X)=1~\text{ for all $r \geq 1$,}~\text{ when $X$ is $\mathcal{UBD}_n$, $\mathcal{LBD}_n$, $\mathcal{TD}_n$, or $\mathcal{C}_n$}.
\end{equation*}

From now on, we consider the cases of $X=\mathcal{T}_n$ or $\mathcal{H}_n$. Since $\mu_r(\mathcal{T}_n) \cong \mu_r(\mathcal{H}_n)$ by an affine linear isomorphism for all $r,n \geq 1$ by Proposition \ref{prop:isomorphism between Toeplitz and Hankel factorization varieties} and the degree is invariant under an affine linear isomorphism, then it suffices to calculate $\deg \mu_r(\mathcal{T}_n)$. We explain how to obtain computational evidence for the degree using numerical algebraic geometry. If $\mu_r(\mathcal{T}_n)=\mathbb{C}^{n \times n}$, then its degree is $1$. Assume that $\mu_r(\mathcal{T}_n)\neq\mathbb{C}^{n \times n}$. By Theorem \ref{thm:dimension for Toeplitz and Hankel structures}, we have that
\begin{equation*}
    \dim \mu_r(\mathcal{T}_n)=2r(n-1)+1.
\end{equation*}
Hence, in order to calculate $\deg \mu_r(\mathcal{T}_n)$ following the definition, we need to choose a generic affine linear subspace of codimension $2r(n-1)+1$ in $\mathbb{C}^{n \times n}$. We can take such a generic subspace of codimension $2r(n-1)+1$ by taking generic $2r(n-1)+1$ affine linear equations on $\mathbb{C}[x_{11},...,x_{nn}]$. Let $L$ be a generic affine linear subspace of codimension $2r(n-1)+1$ defined by generic $2r(n-1)+1$ affine linear equations 
\begin{equation*}
    l_1=...=l_{2r(n-1)+1}=0.
\end{equation*}
Let $U$ be a Zariski open subset of $\mu_r(\mathcal{T}_n)$, which is contained in the image $\mu_r^0(\mathcal{T}_n)=m_r(\mathcal{T}_n^{\times r})$ of the matrix multiplication map $m_r:\mathcal{T}^{\times r} \rightarrow \mathbb{C}^{n \times n}$. Then we can say that 
\begin{equation*}
    L \cap (\mu_r(\mathcal{T}_n) \setminus U)=\emptyset
\end{equation*}
from the genericity of $L$, that is, $L$ only intersects with $U$ in $\mu_r(\mathcal{T}_n)$. In order to count
the number of intersection points counted with multiplicity, we can consider the system of equations
\begin{equation}\label{eq:system of essential polynomial equations}
    l_1(X_1\cdots X_r)=\cdots =l_{2r(n-1)+1}(X_1\cdots X_r)=0
\end{equation}
where $X_1,...,X_r$ are Toeplitz matrices with independent variables. If $r \geq 2$, there are infinitely many solutions of this system because a generic fiber $m_r^{-1}(A)$ of the matrix multiplication map $m_r:\mathcal{T}_n^{\times r} \rightarrow \mu_r(\mathcal{T}_n)$ has the dimension
\begin{equation*}
    \dim m_r^{-1}(A)=r-1
\end{equation*}
by Theorem \ref{thm:dimension for Toeplitz and Hankel structures} and Lemma \ref{lemma:fiber dimension}. Here, note that a generic fiber $m_r^{-1}(A)$ for $A=M_1\cdots M_r \in U$ contains the affine variety
\begin{equation}\label{eq:torus}
    \{(\lambda_1 M, \cdots, \lambda_r M_r)~|~\lambda_1,...,\lambda_r \in \mathbb{C}~\text{ and }~\lambda_1\cdots \lambda_r=1\}.
\end{equation}
We emphasize that if a generic fiber is irreducible, then it must be the same with the affine variety (\ref{eq:torus}), but we still don't know whether a generic fiber is irreducible or not. It implies that it may consist of some components of the form (\ref{eq:torus}). It is known that the number of such components for a generic fiber is fixed, say $\delta$. Hence, if we fix the parameters $\lambda_1,...,\lambda_{r}$ (satisfying $\lambda_1\cdots \lambda_r=1$) on each fiber in the solution set of the system (\ref{eq:system of essential polynomial equations}), then we can calculate $\delta \cdot \deg \mu_r(\mathcal{T}_n)$. In practice, we pick $r-1$ generic homogeneous linear polynomials $l_1',...,l_{r-1}'$ where each $l_i'$ depends only on the variables of $X_i$. Then the equations
\begin{equation*}
    l_1'(X_1)=\cdots=l_{r-1}'(X_{r-1})=1
\end{equation*}
choose only one point on each component of the form (\ref{eq:torus}) in the solution set of the system (\ref{eq:system of essential polynomial equations}). In summary, we will calculate $\delta \cdot \deg \mu_r(\mathcal{T}_n)$ by solving the system of equations
\begin{equation}\label{eq:target system}
    \begin{cases}
        l_1(X_1\cdots X_r)=\cdots =l_{2r(n-1)+1}(X_1\cdots X_r)=0,\\
        l_1'(X_1)=\cdots=l_{r-1}'(X_{r-1})=1
    \end{cases}
\end{equation}
for generic affine linear polynomials $l_1,...,l_{2r(n-1)+1} \in \mathbb{C}[x_{11},...,x_{nn}]$ and $l_1',...,l_{r-1}'$ such that each $l_i'$ depends only on the variables of $X_i$. 

For large $r,n \geq 1$, it is hard to solve the system (\ref{eq:target system}) algebraically. Instead, we can use numerical algebraic geometry. For example, we can estimate the finite solutions of the system (\ref{eq:target system}) using the homotopy continuation method (see \cite{MR2881262}). Without any additional effort, we only can obtain computational evidence for the lower bound of $\delta \cdot \deg \mu_r(\mathcal{T}_n)$, because there would be a solution with multiplicity at least $2$. However, if we have that the solutions all have the multiplicity $1$ by investigating the rank of the Jacobian matrix at each solution, then we can obtain computational evidence for the $\delta \cdot \deg \mu_r(\mathcal{T}_n)$. Furthermore, we can obtain computational evidence for the value $\delta$ for a generic fiber in a similar way.

\begin{example}
    From the Macaulay2 computation, we can obtain computational evidence of
    \begin{equation*}
        \delta\cdot \deg \mu_2(\mathcal{T}_4) = 74~\text{and}~\delta=1,~\text{and so}~\deg \mu_2(\mathcal{T}_4) = 74.
    \end{equation*}
\end{example}

We leave the theoretical determination of the degree of $\mu_r(\mathcal{T}_n)$ for future work. In addition, we also leave the case of $X=\Lambda_n$ for future work.

\subsection{Defining equations via displacement rank}\label{subsection:Defining equations via Displacement}

Let $X$ be an affine variety in $\mathbb{C}^{n \times n}$. A Zariski open subset of $\mu_r(X)$ is parametrized by the matrix multiplication map. Hence, in principle, one can obtain the defining ideal of $\mu_r(X)$ by applying elimination of the parameters (see \cite[Chapter 3]{MR4952933}). 

\begin{example}[Skew-symmetric matrix structure]\label{ex:defining equations for skew-symmetric matrix structure}
    We find the defining ideal of $\mu_2(\Lambda_3)$ and show that $\Lambda_3 \nsubseteq \mu_2(\Lambda_3)$ using the ideal. From the matrix multiplication map $m_2:\Lambda_3^{\times 2} \rightarrow \mathbb{C}^{3 \times 3}$,
    \begin{equation*}
    \begin{aligned}
        &\left(\begin{bmatrix}
            0 & a_{12} & a_{13}\\
            -a_{12} & 0 & a_{23}\\
            -a_{13} & -a_{23} & 0
        \end{bmatrix}, \begin{bmatrix}
            0 & b_{12} & b_{13}\\
            -b_{12} & 0 & b_{23}\\
            -b_{13} & -b_{23} & 0
        \end{bmatrix}\right) \\
        &\mapsto \begin{bmatrix}
            -a_{12}b_{12}-a_{13}b_{13} & -a_{13}b_{23} & a_{12}b_{23}\\
            -a_{23}b_{13} & -a_{12}b_{12}-a_{23}b_{23} & -a_{12}b_{13}\\
            a_{23}b_{12}& -a_{13}b_{12} & -a_{13}b_{13}-a_{23}b_{23}
        \end{bmatrix},
    \end{aligned}
    \end{equation*}
    we have the parametrizations
    \begin{equation*}
        x_{11}=-a_{12}b_{12}-a_{13}b_{13},~ x_{12}=-a_{13}b_{23},~ \cdots,~ x_{33}=-a_{13}b_{13}-a_{23}b_{23},
    \end{equation*}
    and so we obtain the system of $9$ polynomial equations and corresponding ideal $\tilde{I}$ in $\mathbb{C}[a_{ij},b_{ij},x_{ij}~|~1 \leq i,j \leq 3]$. By eliminating the variables $a_{ij}$'s and $b_{ij}$'s in the ideal $\tilde{I}$, we obtain the following defining ideal $I$ of $\mu_2(\Lambda_3)$ in $\mathbb{C}[x_{ij}~|~1 \leq i,j \leq 3]$:
    \begin{equation*}
    \begin{aligned}
        I=\langle & 2x_{13}x_{31}+2x_{23}x_{32}-x_{11}x_{33}-x_{22}x_{33}+x_{33}^2,~2x_{12}x_{31}-x_{11}x_{32}+x_{22}x_{32}+x_{32}x_{33},\\
        &x_{11}x_{31}-x_{22}x_{31}+2x_{21}x_{32}+x_{31}x_{33},~2x_{13}x_{21}-x_{11}x_{23}+x_{22}x_{23}+x_{23}x_{33},\\
        &2x_{12}x_{21}-x_{11}x_{22}+x_{22}^2+2x_{23}x_{32}-x_{22}x_{33},~x_{11}x_{21}+x_{21}x_{22}+2x_{23}x_{31}-x_{21}x_{33},\\
        &x_{11}x_{13}-x_{13}x_{22}+2x_{12}x_{23}+x_{13}x_{33},~x_{11}x_{12}+x_{12}x_{22}+2x_{13}x_{32}-x_{12}x_{33},\\
        &x_{11}^2-x_{22}^2-4x_{23}x_{32}+2x_{22}x_{33}-x_{33}^2\rangle.
    \end{aligned}
    \end{equation*}
    Note that the first generator $2x_{13}x_{31}+2x_{23}x_{32}-x_{11}x_{33}-x_{22}x_{33}+x_{33}^2$ of $I$ does not vanish at the $3 \times 3$ skew-symmetric matrix 
    \begin{equation*}
        \begin{bmatrix}
            0 & 0 & 1\\
            0 & 0 & 0 \\
            -1 & 0 & 0
        \end{bmatrix},
    \end{equation*}
    and so $\Lambda_3 \nsubseteq \mu_2(\Lambda_3)$. 

    Similarly, we can obtain that the defining ideal $I$ of $\mu_2(\Lambda_4)$ contains the polynomial
    \begin{equation*}
        2x_{13}x_{41}+2x_{23}x_{42}-x_{11}x_{43}-x_{22}x_{43}+x_{33}x_{43}+x_{43}x_{44}
    \end{equation*}
    and this polynomial does not vanish at the skew-symmetric matrix 
    \begin{equation*}
        \begin{bmatrix}
            0 & 0 & 1 & 1\\
            0 & 0 & 0 & 0 \\
            -1 & 0 & 0 & 0\\
            -1 & 0 & 0 & 0
        \end{bmatrix}.
    \end{equation*}
    Hence, $\Lambda_4 \nsubseteq \mu_2(\Lambda_4)$. 
\end{example}

Note that the elimination is based on Gr\"obner basis computations, which are often computationally expensive. On the other hand, the \emph{displacement rank approach} has been actively used to study structured matrix computations (see \cite{MR533501, MR1843842}). We will use the approach to obtain some defining equations of $\mu_r(X)$.

\begin{definition}[see {\cite[Section 1.3]{MR1843842}}]
    Let $P,Q \in \mathbb{C}^{n \times n}$. The \emph{displacement operator (of Sylvester type) with respect to $P$ and $Q$}, denoted by $\nabla_{P,Q}$, is defined by
    \begin{equation*}
        \nabla_{P,Q}:\mathbb{C}^{n \times n} \rightarrow \mathbb{C}^{n \times n},~A \mapsto PA-AQ.
    \end{equation*}
    The image $\nabla_{P,Q}(A)$ of $A$ is called the \emph{displacement of $A$ with respect to $P$ and $Q$}, and $\rank(\nabla_{P,Q}(A))$ is called the \emph{displacement rank of $A$ with respect to $P$ and $Q$}. If $P=Q$, then we simply denote $\nabla_{Q}$ instead of $\nabla_{P,Q}$.
\end{definition}

\begin{theorem}\label{thm:lower bound from displacement rank, one matrix}
    Let $Q \in \mathbb{C}^{n \times n}$. Assume that $\rank(\nabla_{Q}(M)) \leq r$ for each $M \in X$. Then, for every $A \in \mathbb{C}^{n \times n}$,
    \begin{equation*}
        \underline{\mathbf{FL}}_X(A) \geq \left\lceil \frac{\rank(\nabla_{Q}(A))}{r} \right\rceil.
    \end{equation*}
\end{theorem}
\begin{proof}
    Let $A \in \mathbb{C}^{n\times n}$. First, we prove $\mathbf{FL}_X(A) \geq \left\lceil \rank(\nabla_{Q}(A))/r \right\rceil$. If $A=M_1M_2$, then
    \begin{equation*}
        \begin{aligned}
            \nabla_{Q}(A)&=Q(M_1M_2)-(M_1M_2)Q\\
            &=(QM_1)M_2-(M_1Q)M_2+M_1(QM_2)-M_1(M_2Q)\\
            &=(QM_1-M_1Q)M_2+M_1(QM_2-M_2Q)\\
            &=\nabla_{Q}(M_1)\cdot M_2+M_1 \cdot \nabla_{Q}(M_2),
        \end{aligned}
    \end{equation*}
    and hence $\rank(\nabla_{Q}(A))\leq \rank(\nabla_{Q}(M_1))+\rank(\nabla_{Q}(M_2))$. Similarly, if $A$ is of $X$-factorization length $k$ so that $A=M_1\cdots M_k$ for $M_1,...,M_k \in X$, we obtain that
    \begin{equation*}
        \rank(\nabla_{Q}(A))=\rank(\nabla_{Q}(M_1\cdots M_k)) \leq \rank(\nabla_{Q}(M_1))+\cdots+\rank(\nabla_{Q}(M_k)) = k \cdot r,
    \end{equation*}
    and hence
    \begin{equation*}
        \mathbf{FL}_X(A) = k \geq \left\lceil \frac{\rank(\nabla_{Q}(M_1\cdots M_k))}{r} \right\rceil.
    \end{equation*}
    Thus, the condition that $\rank(\nabla_{Q}(M_1\cdots M_k)) \leq k \cdot r$ gives some determinantal relations for $\mu_k^0(X)$, and so $\mu_k(X)$ also satisfies the determinantal relations. Therefore, we have $\underline{\mathbf{FL}}_X(A) \geq \lceil \rank(\nabla_{Q}(A))/r \rceil$.
\end{proof}

\begin{theorem}\label{thm:lower bound from displacement rank, two matrices}
     Let $P,Q,R \in \mathbb{C}^{n \times n}$. Assume that $\rank(\nabla_{Q}(M)) \leq r$ for each $M \in X$.
    \begin{itemize}
        \item [(\romannumeral1)] If $\rank(\nabla_{P,Q}(M)) \leq r$ for each $M \in X$, then, for every $A \in \mathbb{C}^{n \times n}$,
    \begin{equation*}
        \underline{\mathbf{FL}}_X(A) \geq \left\lceil \frac{\rank(\nabla_{P,Q}(A))}{r} \right\rceil.
    \end{equation*}
        \item [(\romannumeral2)] If $\rank(\nabla_{Q,R}(M)) \leq r$ for each $M \in X$, then, for every $A \in \mathbb{C}^{n \times n}$,
    \begin{equation*}
        \underline{\mathbf{FL}}_X(A) \geq \left\lceil \frac{\rank(\nabla_{Q,R}(A))}{r} \right\rceil.
    \end{equation*}
    \end{itemize}
\end{theorem}
\begin{proof}
    If $A=M_1M_2$, then 
    \begin{equation*}
        \begin{aligned}
            \nabla_{P,Q}(A)&=P(M_1M_2)-(M_1M_2)Q\\
            &=(PM_1)M_2-(M_1Q)M_2+M_1(QM_2)-M_1(M_2Q)\\
            &=(PM_1-M_1Q)M_2+M_1(QM_2-M_2Q)\\
            &=\nabla_{P,Q}(M_1) \cdot M_2+M_1 \cdot \nabla_{Q}(M_2)
        \end{aligned}
    \end{equation*}
    and
    \begin{equation*}
        \begin{aligned}
            \nabla_{Q,R}(A)&=Q(M_1M_2)-(M_1M_2)R\\
            &=(QM_1)M_2-(M_1Q)M_2+M_1(QM_2)-M_1(M_2R)\\
            &=(QM_1-M_1Q)M_2+M_1(QM_2-M_2R)\\
            &=\nabla_{Q}(M_1) \cdot M_2+M_1 \cdot \nabla_{Q,R}(M_2).
        \end{aligned}
    \end{equation*}
    The rest of the proof proceeds as in the proof of Theorem \ref{thm:lower bound from displacement rank, one matrix}.
\end{proof}

Let $Z_s$ denote the $n \times n$ matrix
\begin{equation*}
    Z_s=\begin{bmatrix}
            0 & & &s\\
            1 & 0 & & \\
            & \ddots & \ddots & \\
            & & 1 & 0
        \end{bmatrix}.
\end{equation*}

\begin{corollary}[Toeplitz matrix structure]\label{cor:displacement rank of Toeplitz}
    Consider $\mathcal{T}_n$ the affine variety of $n \times n $ Toeplitz matrices. Then, for every $A \in \mathbb{C}^{n \times n}$, 
    \begin{equation*}
        \underline{\mathbf{FL}}_{\mathcal{T}_n}(A) \geq \left\lceil \frac{\rank(\nabla_{Z_0}(A))}{2} \right\rceil~\text{ and }~\underline{\mathbf{FL}}_{\mathcal{T}_n}(A) \geq \left\lceil \frac{\rank(\nabla_{Z_1,Z_0}(A))}{2} \right\rceil.
    \end{equation*}
    In particular, for an $n \times n$ matrix $X$ with $n^2$ independent variables, the $(2r+1)$-minors of 
    \begin{equation*}
        \nabla_{Z_0}(X)~\text{ and }~\nabla_{Z_1,Z_0}(X)
    \end{equation*}
    are equations vanishing on $\mu_r(\mathcal{T}_n)$.
\end{corollary}
\begin{proof}
    For a Toeplitz matrix $T$ with $T_{ij}:=t_{j-i}$ for all $1 \le i,j \leq n$, since
    \begin{equation*}
        Z_0T-TZ_0=\begin{bmatrix}
            -t_1 & -t_2 & \cdots & -t_{n-1} & 0\\
            0 & 0 & \cdots & 0 & t_{n-1}\\
            0 & 0 & \cdots & 0 & t_{n-2}\\
            \vdots & \vdots &\ddots & \vdots & \vdots\\
            0 & 0 & \cdots & 0 & t_{1}\\
        \end{bmatrix}
    \end{equation*}
    and 
    \begin{equation*}
        Z_1T-TZ_0=\begin{bmatrix}
            t_{1-n}-t_{1} & t_{2-n}-t_2 & \cdots & t_{-1}-t_{n-1} & t_0\\
            0 & 0 & \cdots & 0 & t_{n-1}\\
            0 & 0 & \cdots & 0 & t_{n-2}\\
            \vdots & \vdots &\ddots & \vdots & \vdots\\
            0 & 0 & \cdots & 0 & t_{1}\\
        \end{bmatrix},
    \end{equation*}
    then we obtain that
    \begin{equation*}
        \rank(\nabla_{Z_0}(T)) \leq 2~\text{ and }~\rank(\nabla_{Z_1,Z_0}(T)) \leq 2.
    \end{equation*}
    By Theorems \ref{thm:lower bound from displacement rank, one matrix} and \ref{thm:lower bound from displacement rank, two matrices}, the assertions hold.
\end{proof}

\begin{example}[Hankel matrix structure]\label{ex:defining equations for Hankel matrix structure}
    We show that $\mathcal{H}_6 \nsubseteq \mu_2(\mathcal{H}_6)$. Consider the Hankel matrix
    \begin{equation*}
        H=\begin{bmatrix}
            0 & 0 & 0 & 0 & 0 & 0\\
            0 & 0 & 0 & 0 & 0 & 1\\
            0 & 0 & 0 & 0 & 1 & 0\\
            0 & 0 & 0 & 1 & 0 & 0\\
            0 & 0 & 1 & 0 & 0 & 0\\
            0 & 1 & 0 & 0 & 0 & 0
        \end{bmatrix} \in \mathcal{H}_6.
    \end{equation*}
    From the proof of Proposition \ref{prop:isomorphism between Toeplitz and Hankel factorization varieties}, we have $\mu_2(\mathcal{H}_6)=\mu_2(\mathcal{T}_6)$, and so it suffices to show that $H \notin \mu_2(\mathcal{T}_6)$. The displacement of $H$  with respect to $Z_1$ and $Z_0$ is 
    \begin{equation*}
        \nabla_{Z_1,Z_0}(H)=Z_1H-HZ_0=\begin{bmatrix}
            0 & 1 & 0 & 0 & 0 & 0\\
            0 & 0 & 0 & 0 & -1 & 0\\
            0 & 0 & 0 & -1 & 0 & 1\\
            0 & 0 & -1 & 0 & 1 & 0\\
            0 & -1 & 0 & 1 & 0 & 0\\
            -1 & 0 & 1 & 0 & 0 & 0
        \end{bmatrix}
    \end{equation*}
    which has rank $6$. By Corollary \ref{cor:displacement rank of Toeplitz}, $\underline{\mathbf{FL}}_{\mathcal{T}_6}(H) \geq 3$, and so $H \notin \mu_2(\mathcal{T}_6)=\mu_2(\mathcal{H}_6)$.
\end{example}

\begin{remark}[Hankel matrix structure]
    Even though $\rank(\nabla_{Z_1,Z_0^T}(H))\leq 2$ for every Hankel matrix $H$ (see \cite[Example 1.3.3]{MR1843842}), we emphasize that
    \begin{equation*}
        \underline{\mathbf{FL}}_{\mathcal{H}_n}(A) \ngeq \left\lceil \frac{\rank(\nabla_{Z_1,Z_0^T}(A))}{2} \right\rceil
    \end{equation*}
    in general. For example, the identity matrix $I_6$ is the square of the reversal permutation matrix $J$ described at (\ref{eq:reversal permutation matrix}), and so is in $\mu_2(\mathcal{H}_6)$. However,
    \begin{equation*}
        \rank(\nabla_{Z_1,Z_0^T}(I_6))=\rank(Z_1-Z_0^T)=5.
    \end{equation*}
\end{remark}

\begin{remark}
    If $n=2m$ for some $m \geq 1$, then a generic matrix $A$ satisfies that $\nabla_{Z_1,Z_0}(A)$ is of full rank. From Theorem \ref{thm:lower bound from displacement rank, two matrices}, we can get the following lower bound:
    \begin{equation*}
        \underline{\mathbf{FL}}_{\mathcal{T}_n}(A) \geq \left\lceil \frac{2m}{2} \right\rceil=m.
    \end{equation*}
    Since $\mathbf{FL}_{\text{gen}}(\mathcal{T}_n)=\lceil 2m/2\rceil+1=m+1$, then we can say that this displacement rank approach does not always yield the optimal lower bound. This implies that this approach does not always provide the whole defining ideal of $\mu_r(\mathcal{T}_n)$.
\end{remark}

We leave the determination of the defining ideals of $\mu_r(\mathcal{T}_n)$ and $\mu_r(\Lambda_n)$ as future work.

\section{Orthogonally Invariant Matrix Structures}\label{section:Orthogonally Invariant Matrix Structures}

Let $\operatorname{O}(n)$ be the $n \times n$ orthogonal group, i.e., $\operatorname{O}(n)=\{g \in \mathbb{C}^{n \times n}~|~g^Tg=I_n\}$. Assume that an affine variety $X \subseteq \mathbb{C}^{n\times n}$ is invariant under the congruent action of $\operatorname{O}(n)$, i.e., $g M g^T \in X$ for any $M \in X$ and $g \in \operatorname{O}(n)$. If $A \in \mathbb{C}^{n \times n}$ satisfies the equality $A=M_1M_2\cdots M_r$ for $M_1,...,M_r \in X$, then 
\begin{equation*}
    gAg^T=g(M_1M_2\cdots M_r)g^T=(gM_1g^T)(gM_2g^T)\cdots(gM_rg^T)
\end{equation*}
where $gM_1g^T,...,gM_rg^T \in X$. Hence, we obtain that
\begin{equation*}
    \mathbf{FL}_X(A)=\mathbf{FL}_X(gAg^T)~\text{ and }~\underline{\mathbf{FL}}_X(A)=\underline{\mathbf{FL}}_X(gAg^T)
\end{equation*}
for any $g \in \operatorname{O}(n)$. The equality for the border $X$-factorization length can be derived from Corollary \ref{cor:limit expression}. In this section, we investigate the irreducible $\operatorname{O}(n)$-invariant subspaces of $\mathbb{C}^{n \times n}$.

It is well-known that $\mathbb{C}^{n \times n}$ has the following irreducible decompositions as $\operatorname{O}(n)$-representation:
\begin{equation*}
    \mathbb{C}^{n \times n}=
        \Lambda_n \oplus \mathcal{S}_n^0 \oplus \mathbb{C}\langle I_n \rangle
\end{equation*}
where $\mathcal{S}_n^0$ denotes the affine variety of $n \times n$ traceless symmetric matrices. Among the irreducible components, an interesting object that has not yet been treated is $\mathcal{S}_n^0$.

First, we consider the case of $n=2$. Since
\begin{equation*}
    \begin{bmatrix}
        a & b\\
        b & -a
    \end{bmatrix}\begin{bmatrix}
        c & d\\
        d & -c
    \end{bmatrix}=\begin{bmatrix}
        ac+bd & ad-bc\\
        -(ad-bc) & ac+bd
    \end{bmatrix},
\end{equation*}
then we can say that a product of two traceless symmetric matrices is a Toeplitz matrix. In addition, since
\begin{equation*}
    \begin{bmatrix}
        a & b\\
        b & -a
    \end{bmatrix}\begin{bmatrix}
        c & d\\
        -d & c
    \end{bmatrix}=\begin{bmatrix}
        ac-bd & ad+bc\\
        ad+bc & -(ac-bd)
    \end{bmatrix},
\end{equation*}
then we can say that a product of three traceless symmetric matrices is a traceless symmetric matrix. Thus, $\mu_r(\mathcal{S}_2^0) \neq \mathbb{C}^{2 \times 2}$ for any $r \geq 1$. However, the situation is different when $n \geq 3$.

\begin{theorem}\label{thm:the generic length of traceless symmetric matrix structure}
    If $n \geq 3$, then  $\mathbf{FL}_{\text{gen}}(\mathcal{S}_n^0)=2$ and $\mathbf{FL}_{\text{max}}(\mathcal{S}_n^0)\leq 10$.
\end{theorem}
\begin{proof}
    We prove that the matrix multiplication map $m_2:(\mathcal{S}_n^0)^{\times 2} \rightarrow \mathbb{C}^{n \times n}$ is dominant. Consider its differential
    \begin{equation*}
        d(m_2)_{(A,B)}: (X,Y) \rightarrow XB+AY
    \end{equation*}
    at a point $(A,B) \in (\mathcal{S}_n^0)^{\times 2}$. By Proposition \ref{prop:the rank of differential and dominance}, it suffices to show that it is surjective. Choose 
    \begin{equation*}
        A=\begin{bmatrix}
            \lambda_1 &  & \\
             & \ddots & \\
             & & \lambda_n
        \end{bmatrix}~\text{ and }~B=\begin{bmatrix}
            \mu_1 &  & \\
             & \ddots & \\
             & & \mu_n
        \end{bmatrix}
    \end{equation*}
    such that $\sum_{i=1}^n \lambda_i=\sum_{i=1}^n \mu_i=0$, $\lambda_i\mu_i$'s are pairwise distinct, and the matrix 
    \begin{equation*}
        \left[\begin{array}{ccc|ccc}
             \mu_1 & & & \lambda_1 & & \\
             & \ddots &  & & \ddots & \\
             & & \mu_n & & & \lambda_n \\ \hline
             1 & \cdots & 1 & & &\\ \hline
             & & & 1 & \cdots & 1
        \end{array}\right]
    \end{equation*}
    has full rank. Note that a generic diagonal matrix pair $(A,B)$ satisfies the conditions, and hence such choice must exist. Let $M \in \mathbb{C}^{n \times n}$. The formula $d(m_2)_{(A,B)}(X,Y)=XB+AY$ implies that 
    \begin{equation*}
        (d(m_2)_{(A,B)}(X,Y))_{ij}=\mu_jX_{ij}+\lambda_iY_{ij}.
    \end{equation*}
    Since $X$ and $Y$ are symmetric, then we also obtain that 
    \begin{equation*}
        (d(m_2)_{(A,B)}(X,Y))_{ji}=\mu_iX_{ji}+\lambda_jY_{ji}=\mu_iX_{ij}+\lambda_jY_{ij}.
    \end{equation*}
    Hence, the following equality must hold:
    \begin{equation}\label{eq:linear system from differential, traceless symmetric structure}
        \begin{bmatrix}
            M_{ij}\\
            M_{ji}
        \end{bmatrix}=\begin{bmatrix}
            \mu_j & \lambda_i\\
            \mu_i & \lambda_j
        \end{bmatrix}\begin{bmatrix}
            X_{ij}\\ Y_{ij}
        \end{bmatrix}.
    \end{equation}
    For all $i,j$ such that $i \neq j$, since 
    \begin{equation*}
        \det \begin{bmatrix}
            \mu_j & \lambda_i\\
            \mu_i & \lambda_j
        \end{bmatrix} = \mu_j\lambda_j-\mu_i\lambda_i \neq 0,
    \end{equation*}
    then there exist $X_{ij}$ and $Y_{ij}$ which satisfy the equation (\ref{eq:linear system from differential, traceless symmetric structure}). In addition, the diagonal entries $M_{11},...,M_{nn}$ must satisfy the following equation:
    \begin{equation*}
        \left[\begin{array}{ccc|ccc}
             \mu_1 & & & \lambda_1 & & \\
             & \ddots &  & & \ddots & \\
             & & \mu_n & & & \lambda_n \\ \hline
             1 & \cdots & 1 & & &\\ \hline
             & & & 1 & \cdots & 1
        \end{array}\right]\begin{bmatrix}
            X_{11} \\ \vdots \\ X_{nn} \\ Y_{11} \\ \vdots \\ Y_{nn}
        \end{bmatrix}=\begin{bmatrix}
            M_{11} \\ M_{22} \\ \vdots \\ M_{nn} \\ 0 \\ 0
        \end{bmatrix}.
    \end{equation*}
    Since the coefficient $(n+2) \times 2n$ matrix is of full rank, then there exist $X_{11},...,X_{nn},Y_{11},...,Y_{nn}$ satisfying this equation. Therefore, $d(m_2)_{(A,B)}$ is surjective, and so $m_2:(\mathcal{S}_n^0)^{\times 2} \rightarrow \mathbb{C}^{n \times n}$ is dominant. 
    
    By Proposition \ref{prop:the number of structured matrices needed to express every matrix}, $\mathbf{FL}_{\text{gen}}(\mathcal{S}_n^0)=2$ implies that every matrix is a product of $8$ traceless symmetric matrices and a diagonal matrix. We claim that a diagonal matrix is a product of two traceless symmetric matrices. If we show that arbitrary diagonal matrix
    \begin{equation*}
        \begin{bmatrix}
            d_1 &  &\\
            & \ddots &\\
            & & d_n
        \end{bmatrix},
    \end{equation*}
    satisfies
    \begin{equation*}
        \begin{bmatrix}
            d_1 &  &\\
            & \ddots &\\
            & & d_n
        \end{bmatrix}=\begin{bmatrix}
            a_1 &  &\\
            & \ddots &\\
            & & a_n
        \end{bmatrix}\begin{bmatrix}
            b_1 &  &\\
            & \ddots &\\
            & & b_n
        \end{bmatrix},~a_1+\cdots+a_n=0,~\text{and }b_1+\cdots+b_n=0,
    \end{equation*}
    then the claim is true. We prove this by considering cases.
    \begin{itemize}
        \item [(\romannumeral1)] Assume that there is at least one $i$ such that $d_i = 0$, say $d_n=0$. Let $a_1,...,a_{n-1}$ be nonzero complex numbers such that $a_1+\cdots+a_{n-1}=0$. Then we have
    \begin{equation*}
        \left[\begin{array}{ccc|c}
            d_1 &  & &\\
            & \ddots & &\\
            & & d_{n-1} &\\ \hline
            & & & 0
        \end{array}\right]=\left[\begin{array}{ccc|c}
            a_1 &  & &\\
            & \ddots & &\\
            & & a_{n-1} & \\ \hline
             & & & 0
        \end{array}\right]\left[\begin{array}{ccc|c}
            d_1/a_1 &  & &\\
            & \ddots & &\\
            & & d_{n-1}/a_{n-1} & \\ \hline
             & & & -\sum_{i=1}^{n-1}(d_i/a_i)
        \end{array}\right],
    \end{equation*}
    which is a product of two traceless diagonal matrices. If $d_n \neq0$, then for $i$ such that $d_i=0$ we can prove similarly.
        \item [(\romannumeral2)] Assume that all $d_i$'s are nonzero. We have
    \begin{equation*}
        \begin{aligned}
           &\left[\begin{array}{ccc|cc}
            d_1 &  & & &\\
            & \ddots & & &\\
            & & d_{n-2} & &\\ \hline
            & & & d_{n-1} & \\
            & & & & d_n
        \end{array}\right]\\
        &=\left[\begin{array}{ccc|cc}
            1 &  & & &\\
            & \ddots & & &\\
            & & 1 & &\\ \hline
             & & & t & \\
             & & & & -(t+n-2)
        \end{array}\right]\left[\begin{array}{ccc|cc}
            d_1 &  & & &\\
            & \ddots & & &\\
            & & d_{n-2} & & \\ \hline
             & & & d_{n-1}/t & \\
             & & & & -d_n/(t+n-2)
        \end{array}\right],
        \end{aligned}
    \end{equation*}
    where $t$ satisfies that $t \neq 0$ and $t \neq -n+2$. The rightmost matrix, say $B(t)$, is traceless if
    \begin{equation*}
        \sum_{i=1}^{n-2}d_i+\frac{d_{n-1}}{t}-\frac{d_n}{t+n-2}=0.
    \end{equation*}
    This equation gives us a quadratic polynomial equation where $t=0$ and $t=-n+2$ are not solutions. Hence, there exists a complex number $t$ such that $t \neq 0$, $t \neq -n+2$, and $B(t)$ is traceless.
    \end{itemize}
    Thus, every diagonal matrix is a product of two traceless symmetric matrices. Therefore, we have that every matrix is a product of $10$ traceless symmetric matrices.
\end{proof}

\section{Alternating Minimization for Structured Matrix Factorization}\label{section:Alternating Minimization for Structured Matrix Factorization}

Although the preceding sections are formulated over $\mathbb{C}$, in this section we work over $\mathbb{R}$ for the numerical computations.

\subsection{Problem formulation}
In this section, we describe an alternating minimization method for constructing Toeplitz factorizations numerically. 
Although the preceding sections are formulated over $\mathbb{C}$, we work over $\mathbb{R}$ in this section, since the examples below concern real target matrices and real Toeplitz factors. 
The same least-squares formulation can also be considered over $\mathbb{C}$.

Given $A\in\R^{n\times n}$ and a prescribed length $r\geq 1$, we consider the optimization problem
\begin{equation}\label{eq:main_obj_mat}
\min_{M_1,\ldots,M_r\in\mathcal{T}_n}
\frac12\|M_1\cdots M_r-A\|_F^2.
\end{equation}
For a Toeplitz matrix $B=[b_{j-i}]$, let
\[
b=(b_{-(n-1)},\ldots,b_{n-1})^T\in\R^{2n-1}
\]
be the vector of its diagonal entries. 
Let $S\in\R^{n^2\times(2n-1)}$ be the Toeplitz selection matrix, so that
\[
\vect(B)=Sb.
\]

For the $k$th block update, fix all factors except $M_k$ and write the variable factor as $B=[b_{j-i}]$. Set
\[
L=M_1\cdots M_{k-1},\qquad R=M_{k+1}\cdots M_r,
\]
with the convention that an empty product is the identity. Since
\[
\vect(LBR)=(R^\top\otimes L)Sb,
\]
the $k$th block minimization reduces to the linear least-squares problem
\begin{equation}\label{eq:block_toeplitz_subproblem}
\min_{b\in\R^{2n-1}}
\frac12\|(R^\top\otimes L)Sb-\vect(A)\|_2^2.
\end{equation}
Thus each block update is linear in the Toeplitz parameters of the factor being updated.

\subsection{Algorithm}

The resulting alternating least-squares method is summarized in Algorithm~\ref{alg:prox-als}.

\begin{algorithm}
\caption{Alternating least-squares updates for Toeplitz factors}
\label{alg:prox-als}
\begin{algorithmic}[1]
\State \textbf{Input:} $A \in \mathbb{R}^{n \times n}$, $r \in \mathbb{N}$, Toeplitz selection matrix $S$
\State initial Toeplitz factors $M_1^{(0)},\dots,M_r^{(0)}$ and $N_{\mathrm{ALS}} \in \mathbb{N}$
\For{$s = 0, \dots, N_{\mathrm{ALS}}-1$}
    \For{$k = 1, \dots, r$}
        \State $L^{(s,k)} = \prod_{i<k} M_i^{(s+1)}$
        \State $R^{(s,k)} = \prod_{i>k} M_i^{(s)}$
        \State $G^{(s,k)} = ((R^{(s,k)})^\top \otimes L^{(s,k)})S$
        \State choose $b_k^{(s+1)} \in \arg\min_{b \in \mathbb{R}^{2n-1}} \frac12\|G^{(s,k)}b-\vect(A)\|_2^2$
        \State $M_k^{(s+1)}$ is the Toeplitz matrix satisfying $\vect(M_k^{(s+1)})=Sb_k^{(s+1)}$
    \EndFor
\EndFor
\State \textbf{return} $(M_1^{(N_{\mathrm{ALS}})},\dots,M_r^{(N_{\mathrm{ALS}})})$
\end{algorithmic}
\end{algorithm}

\begin{theorem}
Let $\{(M_1^{(s)},\dots,M_r^{(s)})\}_{s\ge 0}$ be the iterates produced by Algorithm~\ref{alg:prox-als}. Then
\[
F_s:=\frac12\|M_1^{(s)}\cdots M_r^{(s)}-A\|_F^2
\]
is monotonically nonincreasing. In particular, the residual sequence
\[
\|M_1^{(s)}\cdots M_r^{(s)}-A\|_F
\]
converges.
\end{theorem}

\begin{proof}
Fix a sweep $s$. For $k=1,\ldots,r$, define
\[
B^{(s,k)}=M_1^{(s+1)}\cdots M_k^{(s+1)}M_{k+1}^{(s)}\cdots M_r^{(s)}
\]
and set $B^{(s,0)}=M_1^{(s)}\cdots M_r^{(s)}$. At the $k$th update, all factors other than $M_k$ are fixed, so the objective is the least-squares problem \eqref{eq:block_toeplitz_subproblem}. The chosen minimizer therefore satisfies
\[
\frac12\|B^{(s,k)}-A\|_F^2\leq \frac12\|B^{(s,k-1)}-A\|_F^2.
\]
Iterating this inequality over $k=1,\ldots,r$ gives $F_{s+1}\leq F_s$. Since $F_s\geq 0$, the sequence $\{F_s\}_{s\geq 0}$ has a finite limit, and the residuals are $\sqrt{2F_s}$.
\end{proof}

\begin{remark}
This result concerns only the objective values. In particular, it does not imply that the limiting residual is zero, nor that Algorithm~\ref{alg:prox-als} finds an exact Toeplitz factorization of $A$. It also does not imply convergence of the individual factors without further hypotheses.
\end{remark}

\subsection{Toeplitz examples}
We illustrate the method with two Toeplitz examples. In each example, the ALS output is used as an initial point for a nonlinear least-squares refinement of \eqref{eq:main_obj_mat}.

\begin{example}[A $3\times 3$ Toeplitz example]
The rational Toeplitz factorization from \cite[Example 1]{ye2016every} is
\[
A=\begin{bmatrix}
1 & 2 & 3 \\ 
4 & 5 & 6 \\ 
7 & 8 & 9
\end{bmatrix},\quad
M_1=\begin{bmatrix}
\frac{20}{9} & \frac{8}{9} & -\frac{4}{9} \\ 
\frac{32}{9} & \frac{20}{9} & \frac{8}{9} \\ 
\frac{44}{9} & \frac{32}{9} & \frac{20}{9}
\end{bmatrix},\quad
M_2=\begin{bmatrix}
\frac{1}{4} & 1 & 1 \\ 
1 & \frac{1}{4} & 1 \\ 
1 & 1 & \frac{1}{4}
\end{bmatrix}.
\]
$M_1^{(0)}$ and $M_2^{(0)}$ were initialized at
\[
M_1^{(0)}=\begin{bmatrix}
5 & 4 & 3 \\ 
6 & 5 & 4 \\ 
7 & 6 & 5
\end{bmatrix},
\qquad
M_2^{(0)}=I_3.
\]
The ALS output is then refined by nonlinear least squares. The refined factors are
\[
\widetilde M_1\approx
\begin{bmatrix}
4.67498258 & 1.86999303 & -0.93499652 \\
7.47997212 & 4.67498258 & 1.86999303 \\
10.28496167 & 7.47997212 & 4.67498258
\end{bmatrix},
\]
\[
\widetilde M_2\approx
\begin{bmatrix}
0.11883586 & 0.47534342 & 0.47534342 \\
0.47534342 & 0.11883586 & 0.47534342 \\
0.47534342 & 0.47534342 & 0.11883586
\end{bmatrix}.
\]
These factors recover the factorization above up to the scalar indeterminacy
\[
\widetilde M_1 \approx \alpha M_1,\qquad
\widetilde M_2 \approx \alpha^{-1}M_2,\qquad
\alpha\approx2.10374216.
\]
The resulting residual is
\[
\|\widetilde M_1\widetilde M_2-A\|_F\approx2.09\times 10^{-14}.
\]
\begin{figure}
\includegraphics[width=0.62\textwidth]{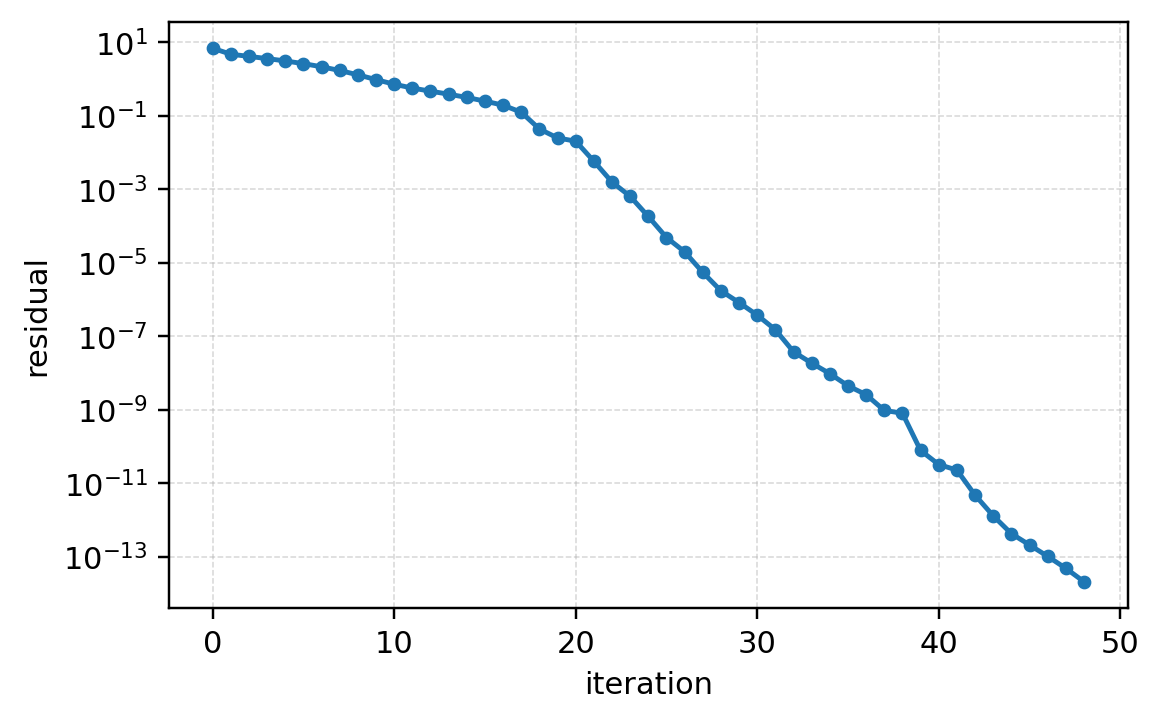}
\caption{Residuals for the alternating minimization method applied to the $3\times 3$ Toeplitz factorization. The residual is shown on a logarithmic scale.}
\label{fig:toeplitz-3x3-residual}
\end{figure}
\end{example}

\begin{example}[A $5\times 5$ Toeplitz example]
For the matrix from \cite[Example 2]{ye2016every},
\[
A=\begin{bmatrix}
2 & 5 & 2 & 5 & 3 \\ 
4 & 5 & 5 & 2 & 2 \\ 
2 & 3 & 2 & 1 & 5 \\ 
3 & 1 & 5 & 2 & 3 \\ 
4 & 1 & 2 & 4 & 3
\end{bmatrix},
\]
the construction displayed there expresses the target as a product of Toeplitz factors together with permutation matrices. 
We choose $r=15$ because the general upper bound $2n+5$ in Theorem~\ref{thm:previous results}.(\romannumeral2) gives $2\cdot 5+5=15$ for $n=5$. 
Thus this experiment tests whether the alternating minimization procedure can numerically realize a Toeplitz-only factorization at the length guaranteed by the theory.
We apply the alternating minimization method with $r=15$.
The computed Toeplitz factors $M_1,\ldots,M_{15}$ satisfy
\[
\|M_1\cdots M_{15}-A\|_F\approx6.91\times 10^{-14}.
\]
Thus, the computation gives a numerical Toeplitz-only factorization of the same target.
\begin{figure}
\includegraphics[width=0.62\textwidth]{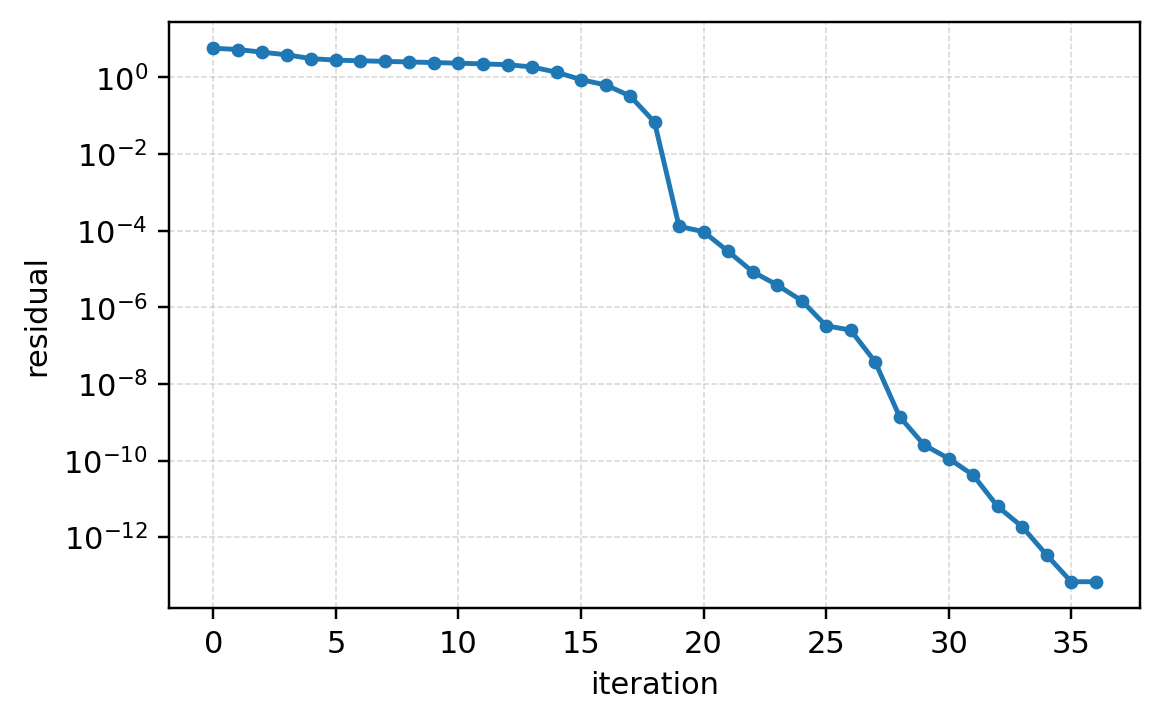}
\caption{Residuals for the alternating minimization method applied to the $5\times 5$ target matrix from \cite[Example 2]{ye2016every} with $r=15$. The residual is shown on a logarithmic scale.}
\label{fig:toeplitz-5x5-residual-depth15}
\end{figure}
The next display records the 15 factors on a common color scale, with five factors in each row.
\begin{figure}
\includegraphics[width=0.98\textwidth]{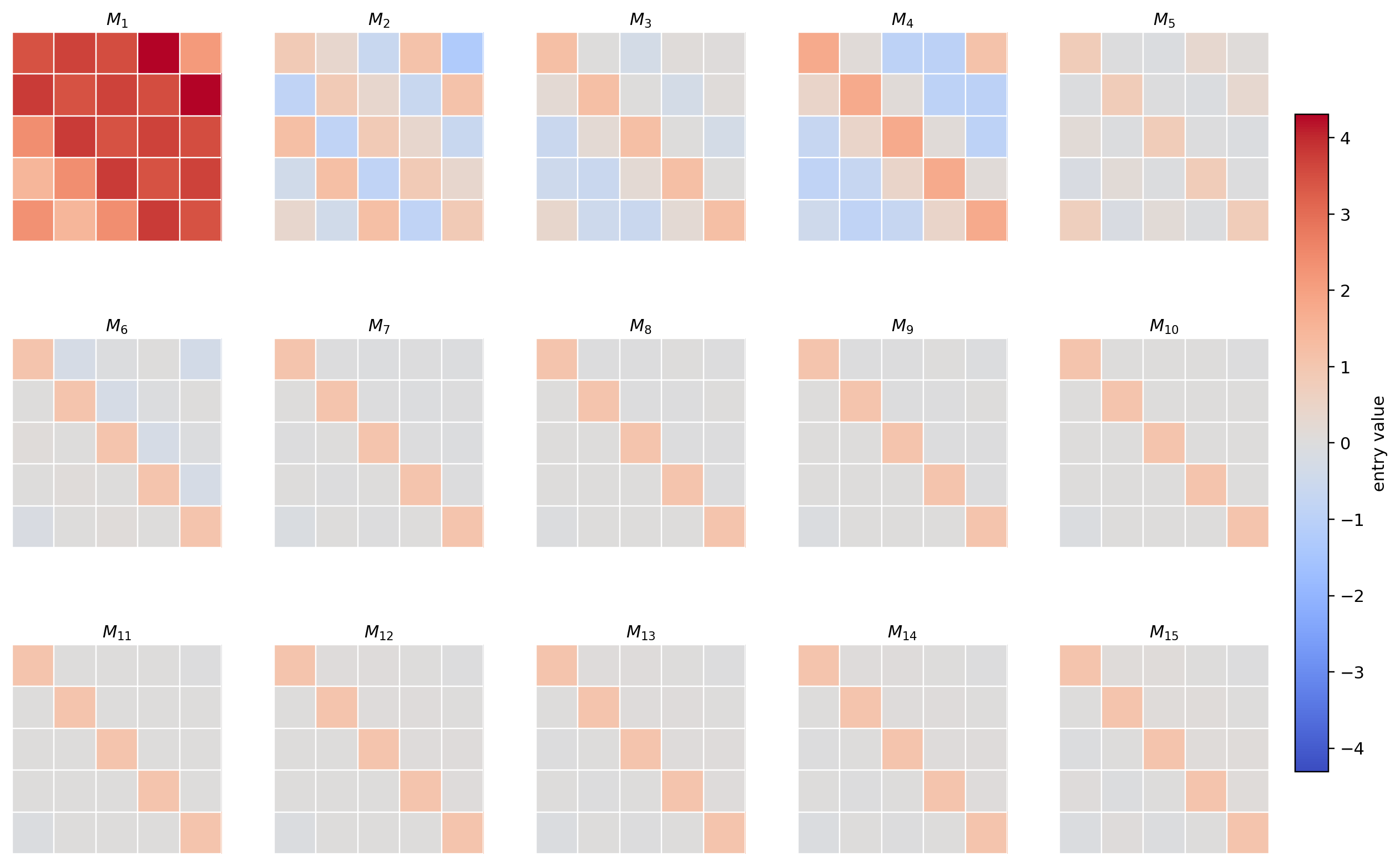}
\caption{Computed Toeplitz factors $M_1,\ldots,M_{15}$ for the $5\times 5$ example. All factors are displayed on the same color scale.}
\label{fig:toeplitz-5x5-factors-depth15}
\end{figure}
\end{example}

\section{Future Directions}

We suggest several directions for further study. First, we have left the theoretical determination of the degree and defining ideal of $\mu_r(\mathcal{T}_n)$ and $\mu_r(\Lambda_n)$ in Section \ref{subsection:Degree} and \ref{subsection:Defining equations via Displacement} for future work. Second, if there is an interesting matrix structure with which we have not dealt in this paper, then investigating the corresponding (border) structured matrix factorization length and the corresponding factorization variety will be also meaningful. Let $X \subseteq \mathbb{C}^{n \times n}$ denote an affine variety. Third, it will be interesting to calculate the Euclidean distance degree (simply, ED degree) of an $r$-th $X$-factorization variety. The ED degree measures the number of critical points of the squared distance function from a general matrix (see \cite{draisma2016euclidean}). Hence, it is related to the low-length approximation problem: for a given matrix, what is the closest matrix whose $X$-factorization length is at most $r$? Fourth, one should investigate the stability of $X$-factorization length under small perturbations, as in the cases of matrix rank, eigenvalue and singular value.  Fifth, since we did not suggest any deterministic and efficient algorithm for the structured matrix factorization length, it is natural to study the computational complexity of determining this quantity. In particular, it will be interesting for complexity theorists to determine whether this problem is in a specific complexity class, for instance, the class NP. Sixth, we can consider the (border) structured matrix factorization length over $\mathbb{R}$ instead of $\mathbb{C}$. It can be more practical considering the applicability to real-life problems. Since $\mathbb{R}$ is not algebraically closed, this direction will require the spirit of real algebraic geometry or differential geometry. Finally, one can generalize the theory from a single structure to several structures. Instead of fixing a single affine variety $X \subseteq \mathbb{C}^{n \times n}$, one can consider several affine varieties $X_1,...,X_k \subseteq \mathbb{C}^{n \times n}$ and study the matrix multiplication map $m_k:X_1 \times \cdots \times X_k \rightarrow \mathbb{C}^{n \times n}$. This framework includes the case of generalized Vandermonde factorization described in \cite[Section 6.2]{ye2017new}.

\section*{Data availability}
Our code is available at \url{https://github.com/taehyeong-matrix/Factorization-Length}

\bibliographystyle{abbrvnat}
\bibliography{bibliography}

\end{document}